\documentclass[12pt,oneside]{amsart}
\pdfoutput=1

\usepackage{amssymb,amsthm} 
\usepackage{enumerate, amsfonts, latexsym,epsfig, color, hyperref}
\usepackage{epstopdf,url}

\input xy
\xyoption{all}

\usepackage{thm-restate}

\newtheorem {theorem}{Theorem}[section]
\newtheorem {lemma} [theorem] {Lemma}
\newtheorem {proposition} [theorem] {Proposition}

\newtheorem {corollary} [theorem] {Corollary}

\newtheorem {question} {Question}

\newtheorem {claim} {Claim}[theorem]
\newtheorem {assumption}[theorem]{Standing Assumption}

\theoremstyle{definition}
\newtheorem {definition} [theorem] {Definition}
\newtheorem {example} [theorem] {Example}
\newtheorem {remark} [theorem] {Remark}
\newtheorem {terminology} [theorem] {Terminology}
\newtheorem {notation} [theorem] {Notation}
\def\co{\colon\thinspace}

\def\R {\mathbb R}

\def\Z {\mathbb Z}
\def\mcP {\mathcal P}

\def\Isom {\mathrm{Isom}}
\def\mc {\mathcal}

\def\Aut {\mathrm{Aut}}
\def\ker {\mathrm{ker}}
\def\stab {\mathrm{Stab}}

\newcommand{\bZ}{\mathbb{Z}}
\newcommand{\bR}{\mathbb{R}}
\newcommand{\bH}{\mathbb{H}}
\newcommand{\bN}{\mathbb{N}}
\newcommand{\acts}{\curvearrowright}
\newcommand{\Pred}{\mc{P}^{\mathrm{red}}}
\newcommand{\barG}{\bar{G}}
\newcommand{\barP}{\bar{\mc{P}}}

\newcommand{\MultiEnded}{\mc{M}}
\newcommand{\ad}{\mathrm{ad}}

\begin{document}

\title[Dehn fillings and splittings]{Dehn fillings and elementary splittings}

\author[D. Groves]{Daniel Groves}
\address{Department of Mathematics, Statistics, and Computer Science,
University of Illinois at Chicago,
322 Science and Engineering Offices (M/C 249),
851 S. Morgan St.,
Chicago, IL 60607-7045}
\email{groves@math.uic.edu}

\author[J.F. Manning]{Jason Fox Manning}
\address{Department of Mathematics, 310 Malott Hall, Cornell University, Ithaca, NY 14853}
\email{jfmanning@math.cornell.edu}


\thanks{The results in this paper were instigated at the Mathematisches Forschungsinstitut Oberwolfach in June, 2011, and we thank the MFO for its hospitality.  Both authors were supported in part by the NSF (under grants DMS-0953794 and DMS-1462263), who we thank for their support.}

\begin{abstract}
We consider conditions on relatively hyperbolic groups about the non-existence of certain kinds of splittings, and show these properties persist in long Dehn fillings.  We deduce that certain connectivity properties of the Bowditch boundary persist under long fillings.
\end{abstract}
\maketitle

\section{Introduction}
Thurston's Hyperbolic Dehn Filling Theorem \cite[Section 5.8]{thurston:notes} 
shows that sufficiently long (topological) Dehn fillings of a
$1$--cusped hyperbolic manifold are closed hyperbolic manifolds.  In
particular the fundamental groups of these fillings are \emph{one-ended} and \emph{word hyperbolic}.  Thurston's argument is to deform the
hyperbolic structure on the cusped manifold to one whose completion
is the filled manifold.  Gromov and Thurston's \emph{$2\pi$-Theorem} \cite{bleilerhodgson:twopi} makes the hypothesis of ``sufficiently long'' more quantitative, concluding that the filled manifold is negatively curved.  Agol and Lackenby's \emph{$6$-Theorem} \cite{agol:whds, lackenby:whds}
shows that the group-theoretic conclusions can be obtained by a
softer, more combinatorial argument.  This work was part of the inspiration for results
about purely group-theoretic Dehn filling obtained by Osin
\cite{osin:peripheral} and the authors \cite{rhds}, and generalized
still further in the work of Dahmani--Guirardel--Osin \cite{DGO}.
These results all have a ``hyperbolic-like'' conclusion analogous to that of the $6$-Theorem.  However, none say anything
about one-endedness of the quotient.  The following result remedies this.

\begin{theorem} \label{cor to main thm}
 Suppose that $(G,\mcP)$ is relatively hyperbolic, with $\mcP$ consisting of virtually polycyclic subgroups.  If $G$ does not admit any nontrivial elementary
 splittings then sufficiently long Dehn fillings of $G$ do not admit any nontrivial elementary splittings.
\end{theorem}
Let us clarify some terminology.  In this paper, when we say that $(G,\mcP)$ is \emph{relatively hyperbolic} (Definition \ref{def:RH}), we always assume $G$ is finitely generated, and that no $P\in \mcP$ is equal to $G$.  We do not assume that the elements of $\mcP$ are infinite or non-relatively hyperbolic.  An \emph{elementary} subgroup (Definition \ref{def:elementary}) is one which is either virtually cyclic or parabolic.  For \emph{sufficiently long Dehn fillings} see Definition \ref{def:fill}.

In Subsection \ref{ss:classical} we spell out the connection between Theorem \ref{cor to main thm} and Thurston's Hyperbolic Dehn Filling Theorem.  In particular we can deduce the one-endedness of fundamental groups of sufficiently long classical fillings of a hyperbolic $3$--manifold without using the fact they are nonpositively curved manifolds.

 Example \ref{ex:hypotheses are required} shows that in our more general setting it is not enough to assume that $G$ is one-ended in order
to infer that long fillings are one-ended (even when the elements of $\mc{P}$ are abelian).
Theorem \ref{cor to main thm} follows from a much more general result, Theorem \ref{t:easy maintheorem}, which we state in Subsection \ref{ss:Statements} below.

Another of our motivations is to understand the relationship between the
Bowditch boundary of a relatively hyperbolic group and the Bowditch boundary of long
Dehn fillings (or, in case the filled group is hyperbolic, the Gromov boundary).  In the case
of classical Dehn filling, the Bowditch boundary of the original group is $S^2$, as is the 
Gromov boundary of the filled group.  In Section \ref{s:boundaries}, we review work of Bowditch and others on the close relationship between connectedness properties of the boundary of a relatively hyperbolic group and (non-)existence of elementary splittings of the group.  Using these results, we prove:

\begin{restatable}{theorem}{nolocalcutpoints} \label{t:no local cutpoints}
  Suppose that $(G,\mc{P})$ is relatively hyperbolic, with $\mc{P}$ consisting of virtually polycyclic groups.  Suppose that $\partial(G,\mc{P}^{\mathrm{red}})$ is connected with no local cut points.  

Then for all sufficiently long fillings $(\bar{G},\bar{\mc{P}})$, we have $\partial(\bar{G},\bar{\mc{P}}^{\mathrm{red}})$  connected with no local cut points.
\end{restatable}
The peripheral structure $\mc{P}^{\mathrm{red}}$ is found by discarding the hyperbolic subgroups of $\mc{P}$.  This can be done without affecting the relative hyperbolicity of the group pair.  In the case of classical Dehn filling, it corresponds to considering the filled group as a hyperbolic group, and considering its Gromov boundary.
See Section \ref{s:boundaries} for more details.

\subsection{Main results} \label{ss:Statements}

We now proceed to give a description of the more technical Theorem \ref{t:easy maintheorem} and show why it suffices to prove Theorem \ref{cor to main thm}.  We also state a related result, Theorem \ref{t:maintheorem}, and give another application to the Bowditch boundary.

\begin{definition}\label{def:fill}
 Suppose that $G$ is a group and $\mc{P} = \{ P_1 ,\ldots , P_n \}$ is a collection of subgroups.  A {\em Dehn filling} (or just {\em filling}) of $(G,\mc{P})$ is a quotient map: $\phi \co G \to G/K$, where $K$ is the normal closure in $G$ of some collection $K_i \unlhd P_i$.  We write
 \[	G/K = G(K_1, \ldots, K_n)	\]
 for this quotient.  The subgroups $K_1, \ldots, K_n$ are called the {\em filling kernels}.  We also write $\phi\co (G,\mcP)\to (\bar{G},\bar{\mcP})$, where $\mcP$ is the collection of images of the $P\in \mcP$.
 
 We say that a property holds {\em for all sufficiently long fillings of $(G,\mc{P})$} if there is a finite set $\mc{B} \in G \smallsetminus \{ 1 \}$ so that whenever $K_i \cap \mc{B} = \emptyset$ for all $i$, the group $G/K$ has the property.
\end{definition}
See Subsection \ref{ss:Dehn} for more details.  See Definition \ref{def:RH} for the definition of what it means for $(G,\mc{P})$ to be a relatively hyperbolic group pair, and see Theorem \ref{t:RHDF} for the main result of relatively hyperbolic Dehn filling.

\begin{definition}
 A group $H$ is {\em small} if $H$ has no subgroup isomorphic to a nonabelian free group.
 
 A group $H$ is {\em slender} if every subgroup of $H$ is finitely generated.
\end{definition}

\begin{definition}
  Let $(G,\mcP)$ be a group pair and let $\MultiEnded$ be the class of all finitely generated groups with more than one end. 
    We say that a filling of $(G,\mc{P})$ is {\em $\MultiEnded$-finite} if for all $P \in \MultiEnded \cap \mc{P}$, the associated filling kernel $K \unlhd P$ has finite index in $P$.
 
 We say that a filling of $(G,\mc{P})$ is {\em co-slender} if for each $P \in \mc{P}$ with associated filling kernel $K$, the group $P/K$ is slender.
\end{definition}

The following is the first main result of this paper.

\begin{restatable}{theorem}{easymaintheorem} \label{t:easy maintheorem}
 Let $G$ be a group which is hyperbolic relative to a finite collection $\mc{P}$ of subgroups, and suppose that all small subgroups of $G$ are finitely generated.  Furthermore, suppose that $G$ admits no nontrivial elementary splittings. Then all sufficiently long $\MultiEnded$--finite co-slender fillings $(G,\mc{P}) \to (\bar{G},\bar{\mc{P}})$ have the property that $\bar{G}$ admits no nontrivial elementary splittings.
\end{restatable}
\begin{remark}
  For $G$ hyperbolic, or $(G,\mcP)$ relatively hyperbolic with $\mcP$ consisting of slender groups, the hypothesis that small subgroups are finitely generated holds (see Lemma \ref{lem:small is fg} below).
\end{remark}

\begin{proof}[Proof of \ref{cor to main thm} from \ref{t:easy maintheorem}]
By Lemma \ref{lem:small is fg} below, if all elements of $\mc{P}$ are virtually polycyclic (or more generally slender) and $(G,\mc{P})$ is relatively hyperbolic, then all small subgroups of $G$ are finitely generated.  Moreover, in this case the multi-ended elements of $\mc{P}$ are all two-ended.  Suppose some $P_i\in\mcP$ is two-ended.  Then for all sufficiently long fillings $G(K_1,\ldots,K_n)$, the corresponding filling kernel $K_i$ is either trivial or finite index.  Letting $\mcP' = \mcP\smallsetminus\{P_i\}$, the pair $(G,\mcP')$ is also relatively hyperbolic, and has an elementary splitting if and only if $(G,\mcP)$ does.
Thus all sufficiently long fillings of $(G,\mcP)$ are $\MultiEnded$--finite fillings of some $(G,\mcP')$ with $\mcP'\subseteq \mcP$.
  Moreover, any quotient of a virtually polycyclic group is slender, so the assumption of the filling being co-slender is also satisfied.  Thus, Theorem \ref{cor to main thm} follows from Theorem \ref{t:easy maintheorem}.
\end{proof}

Theorem \ref{t:easy maintheorem} concerns groups which do not admit {\em any} elementary splittings.  However, many one-ended relatively hyperbolic groups admit some elementary splittings over virtually cyclic groups, but no splittings over parabolic subgroups.  It is natural to ask if the non-existence of splittings over parabolic subgroups persists under long fillings.
In this direction, the second main result of this paper is the following.  

\begin{restatable}{theorem}{maintheorem} \label{t:maintheorem}
 Let $G$ be a group which is hyperbolic relative to a finite collection $\mc{P}$ of subgroups, and suppose that all small subgroups of $G$ are finitely generated.  Furthermore, suppose that $G$ is one-ended and admits no proper peripheral splittings.  Then all sufficiently long $\MultiEnded$--finite co-slender fillings $(G,\mc{P}) \to (\bar{G},\bar{\mc{P}})$ have the property that $\bar{G}$ is one-ended and admits no splittings over parabolic subgroups.
\end{restatable}
See Subsection \ref{ss:Peripheral Splittings} for the definition of a proper peripheral splitting.  We note here that if the filled groups in the conclusion admit no splittings over parabolic subgroups then they admit no proper peripheral splittings.

Like Theorem \ref{t:easy maintheorem}, Theorem \ref{t:maintheorem} has consequences for the Bowditch boundary.  Again the result is easiest to state for $(G,\mcP)$ with $\mcP$ consisting of virtually polycyclic groups.

\begin{restatable}{theorem}{nocutpoints}  \label{cor:nocutpoints}
  Suppose that $(G,\mc{P})$ is relatively hyperbolic, with $\mc{P}$ consisting of virtually polycyclic groups.  Suppose further that the Bowditch boundary $\partial(G,\mc{P})$ is connected with no cut point.

  Then for all sufficiently long $\MultiEnded$--finite fillings $(G,\mc{P}) \twoheadrightarrow (\bar{G},\bar{\mc{P}})$, the resulting boundary $\partial(\bar{G},\bar{\mc{P}}^{\infty})$
is connected and has no cut points.
\end{restatable}
Note that having $\bar{G}$ one-ended is a stronger condition than having the (reduced) Bowditch boundary of $(\bar{G}, \bar{\mc{P}})$ connected.  For example, consider the relatively hyperbolic pair $(F,\{ C \})$ where $F$ is the fundamental group of a once-punctured torus and $C$ is the cyclic group corresponding to the puncture.

\begin{question}
  For $(G,\mc{P})$ relatively hyperbolic with $\partial(G,\mc{P})$ connected without cut points, 
  is it true that $\partial(\barG,\barP^{\infty})$ is connected without cut points, for all sufficiently long fillings?
   If not, what hypotheses weaker than slenderness and tameness are required?\footnote{See Section \ref{s:boundaries} for the definition of tame.}
\end{question}

\begin{question}
  Our proofs are by contradiction, using limiting arguments.  Are there ``effective'' versions of our results?
\end{question}

\subsection{Classical Dehn filling}\label{ss:classical}
\begin{definition}
  Let $M^n$ be a manifold without boundary whose ends are all homeomorphic to $T^{n-1}\times \bR$.
  Remove neighborhoods of the ends $E_1,\ldots,E_k$, leaving a manifold $\bar{M}$ with boundary homeomorphic to a disjoint union of $k$ manifolds, each homeomorphic to $T^{n-1}$.  To each of these boundary components $B_i$ glue a copy of $D^2\times T^{n-2}$ by some homeomorphism $\phi_i\co \partial (D^2\times T^{n-2})\to B_i$.  We'll call the result a \emph{classical Dehn filling} of $M$.\end{definition}

The following corollary of Theorem \ref{cor to main thm} makes explicit the connection to classical Dehn filling and the $6$-Theorem.  (This corollary also follows from the fact that the fillings support Riemannian metrics of nonpositive curvature \cite{schroeder:cuspclosing}, so the universal cover is $\bR^n$.)
\begin{corollary}
  Let $n\geq 3$.  All sufficiently long classical Dehn fillings of a finite-volume hyperbolic $n$--manifold have one-ended fundamental group.
\end{corollary}
\begin{proof}
  Let $M$ be such a manifold, $G = \pi_1M$, and let $\mc{P}$ be the collection of fundamental groups of ends.  The elements of $\mcP$ are free abelian of rank at least $2$.  In particular they are one-ended and virtually polycyclic.
  The Bowditch boundary of $(G,\mc{P})$ can be identified with $S^{n-2}$, so it is connected and has no local or global cut point.  Work of Bowditch (see Corollary \ref{cor:elementarysplitting}) then implies that $G$ has no elementary splitting.  Theorem \ref{cor to main thm} implies that sufficiently long fillings are $1$-ended.
\end{proof}

Another group-theoretic consequence of Thurston's Hyperbolic Dehn Filling Theorem is that long fillings are torsion-free.  In Section \ref{sec:torsion} we analyze how torsion behaves under Dehn filling.  In particular,  Theorem \ref{thm:torsion in fillings} implies that for torsion-free relatively hyperbolic groups, long enough fillings where each $P_i/K_i$ is torsion-free are torsion-free.  It follows that fundamental groups of sufficiently long classical Dehn fillings of finite-volume hyperbolic manifolds are torsion-free.  (This result also follows from the above-cited result of Schroeder.)

\subsection{Cautionary Examples}

We collect here some examples to illustrate the importance of the hypotheses in Theorem \ref{t:maintheorem}.

\begin{example}
 Except in this example, when $(G,\mcP)$ is relatively hyperbolic, we assume no $P\in \mcP$ is equal to $G$ (in some sources, the pair is then said to be \emph{properly} relatively hyperbolic).
 The following shows the necessity of this assumption in Theorem \ref{t:maintheorem}.

 The group $G = \Z \oplus \Z$ (considered as a relatively hyperbolic group pair $(G,\{G\})$) admits longer and longer fillings onto $\Z$, which admits a free splitting.  But in this case $G$ is one-ended and admits no proper peripheral splitting.
\end{example}

\begin{example} \label{ex:hypotheses are required}
This example shows that the hypothesis of $\MultiEnded$--fi\-nite\-ness is necessary in Theorem \ref{t:maintheorem}.

Let $S$ be a closed orientable surface of genus $2$ and let $G=\pi_1(S)$.  Let $\gamma$ be a waist curve of $S$ and let $P_1$ be cyclic, generated by $\gamma$.  Any nontrivial filling of $G$ along the peripheral structure $\mc{P} = \{P_1\}$ gives a group which splits nontrivially over a finite group.

Of course the pair $(G,\mc{P})$ admits a proper peripheral splitting \[ G = F_2\ast_{P_1}P_1\ast_{P_1}F_2.\]  However, let $\gamma'$ be another simple closed curve so that $\gamma \cup \gamma'$ fills $S$, and let $P_2$ be the cyclic subgroup generated by $\gamma'$.  Let $\mc{P}' = \{ P_1 , P_2 \}$.  It is not hard to see that $(G,\mc{P}')$ has no proper peripheral splitting.  

The fillings $G(\langle \gamma^k \rangle, \{ 1 \})$ all split nontrivially over finite groups.

Note that this example shows that 
\begin{enumerate}
\item The condition that the filling is $\MultiEnded$-finite is required; and
\item It is not enough to assume that $G$ is one-ended for Theorem \ref{t:maintheorem}; the hypothesis of having no nontrivial peripheral splittings is also needed.
\end{enumerate}
\end{example}

\begin{example} \label{ex:slender required}
  The following example shows that the hypotheses of $\MultiEnded$--finite and co-slender can't be weakened just to every filling kernel being infinite.  Let $G = \pi_1\Sigma$ as above, and consider two simple closed curves $\alpha, \beta$ which fill $\Sigma$.  Choose a crossing point $p\in \alpha\cap\beta$ at which to base the fundamental group.  Then $P = \langle \alpha,\beta\rangle$ is the $\pi_1$--image of some immersed torus $T$ with a single boundary component.  We claim:
  \begin{enumerate}
  \item There are $\alpha$ and $\beta$ so that $P$ is malnormal in $G$.
  \item For any action of $G$ on a tree $T$ without global fixed point, some element (either $\alpha$ or $\beta$) of $P$ acts hyperbolically on $T$.
  \end{enumerate}
  The first condition ensures $(G,\{P\})$ is relatively hyperbolic \cite[Theorem 7.11]{bowditch:relhyp}.  The second implies there is no proper peripheral splitting.  But by choosing fillings $K_n\lhd P$ with $K_n$ normally generated by $\alpha^n$, we obtain arbitrarily long fillings of $(G,\{P\})$ which split over finite groups.

In order to obtain malnormality of $\langle \alpha,\beta \rangle$, one can choose $\alpha$ and $\beta$ such that there is a negatively curved cone metric on $\Sigma$ so that $\alpha$ and $\beta$ are represented by (local) geodesics which meet at a point $p$ with four angles of $\pi$.  This defines an immersion from the rose with two petals into $\Sigma$ which is locally convex.  Any elevation to the universal cover is therefore convex.  Different elevations meet at single points, from which it easily follows that the subgroup $P$ is malnormal.
\begin{figure}[htbp]
\begingroup%
  \makeatletter%
  \providecommand\color[2][]{%
    \errmessage{(Inkscape) Color is used for the text in Inkscape, but the package 'color.sty' is not loaded}%
    \renewcommand\color[2][]{}%
  }%
  \providecommand\transparent[1]{%
    \errmessage{(Inkscape) Transparency is used (non-zero) for the text in Inkscape, but the package 'transparent.sty' is not loaded}%
    \renewcommand\transparent[1]{}%
  }%
  \providecommand\rotatebox[2]{#2}%
  \ifx\svgwidth\undefined%
    \setlength{\unitlength}{386.45335693bp}%
    \ifx\svgscale\undefined%
      \relax%
    \else%
      \setlength{\unitlength}{\unitlength * \real{\svgscale}}%
    \fi%
  \else%
    \setlength{\unitlength}{\svgwidth}%
  \fi%
  \global\let\svgwidth\undefined%
  \global\let\svgscale\undefined%
  \makeatother%
  \begin{picture}(1,0.47055493)%
    \put(0,0){\includegraphics[width=\unitlength,page=1]{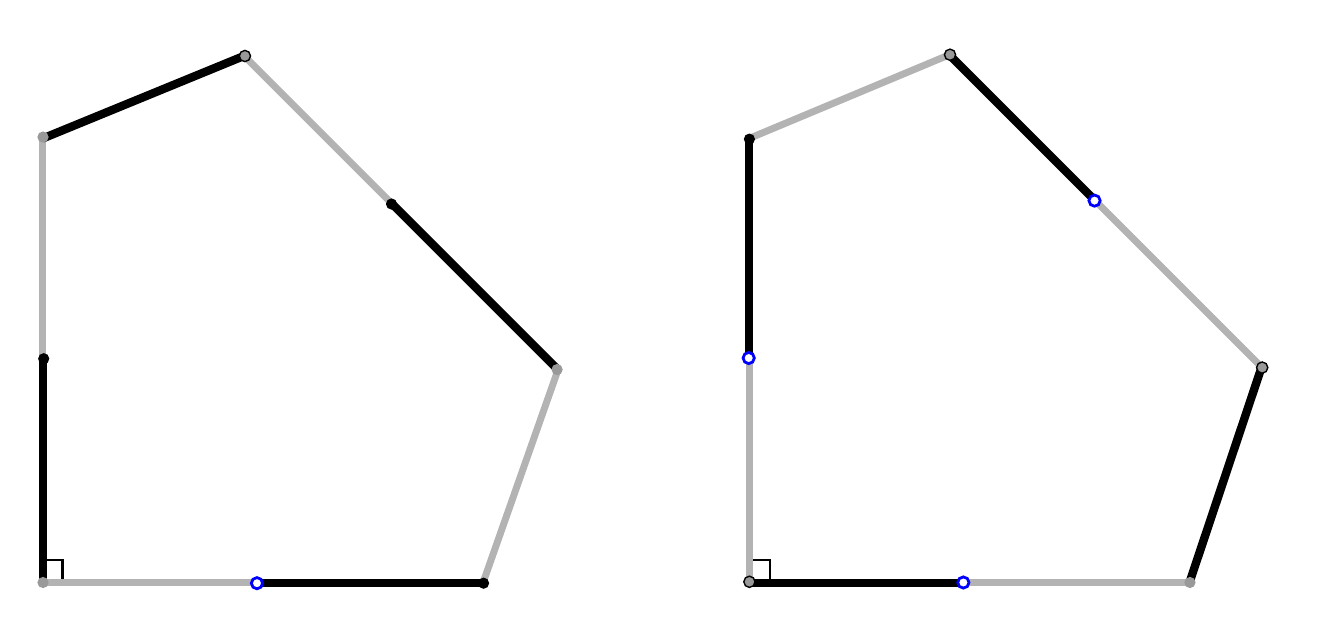}}%
    \put(0.17833305,0.44280951){\color[rgb]{0,0,0}\makebox(0,0)[lb]{\smash{}}}%
    \put(0.18720496,0.00365099){\color[rgb]{0,0,0}\makebox(0,0)[lb]{\smash{$*$}}}%
    \put(0.70916777,0.00512966){\color[rgb]{0,0,0}\makebox(0,0)[lb]{\smash{$*$}}}%
    \put(0.5228581,0.19143933){\color[rgb]{0,0,0}\makebox(0,0)[lb]{\smash{$*$}}}%
    \put(0.82893829,0.33191092){\color[rgb]{0,0,0}\makebox(0,0)[lb]{\smash{$*$}}}%
    \put(-0.00206202,0.19883248){\color[rgb]{0,0,0}\makebox(0,0)[lb]{\smash{$1$}}}%
    \put(0.36612137,0.00956553){\color[rgb]{0,0,0}\makebox(0,0)[lb]{\smash{$1$}}}%
    \put(0.53172997,0.37922755){\color[rgb]{0,0,0}\makebox(0,0)[lb]{\smash{$1$}}}%
    \put(0.29810354,0.3319108){\color[rgb]{0,0,0}\makebox(0,0)[lb]{\smash{$1$}}}%
    \put(0.01272445,0.0110442){\color[rgb]{0,0,0}\makebox(0,0)[lb]{\smash{$2$}}}%
    \put(0.88512685,0.0110442){\color[rgb]{0,0,0}\makebox(0,0)[lb]{\smash{$2$}}}%
    \put(0.42230999,0.1899606){\color[rgb]{0,0,0}\makebox(0,0)[lb]{\smash{$2$}}}%
    \put(0.00828852,0.38218488){\color[rgb]{0,0,0}\makebox(0,0)[lb]{\smash{$2$}}}%
    \put(0.18129035,0.45168138){\color[rgb]{0,0,0}\makebox(0,0)[lb]{\smash{$3$}}}%
    \put(0.53172997,0.00808693){\color[rgb]{0,0,0}\makebox(0,0)[lb]{\smash{$3$}}}%
    \put(0.9546234,0.19143927){\color[rgb]{0,0,0}\makebox(0,0)[lb]{\smash{$3$}}}%
    \put(0.7002959,0.44872408){\color[rgb]{0,0,0}\makebox(0,0)[lb]{\smash{$3$}}}%
  \end{picture}%
\endgroup%

  \caption{Gluing up the Euclidean polygons as shown gives a genus $2$ surface with a non-positively curved metric.  The polygons can be made slightly hyperbolic (preserving the angles of $\pi$ and $\frac{\pi}{2}$) to get a negatively curved metric.  The black and grey loops based at the vertex $(*)$ generate a malnormal free subgroup of rank $2$.}
  \label{fig:surface}
\end{figure}
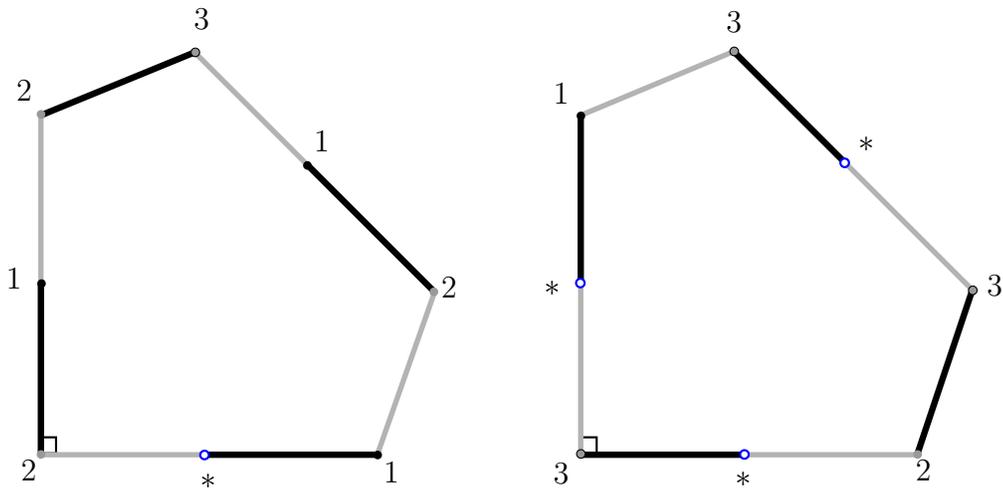
See Figure \ref{fig:surface} for a specific example.
\end{example}

This shows that Theorem \ref{t:maintheorem} does not hold without the hypothesis the fillings are $\MultiEnded$-finite and/or co-slender.  However, in that example, the boundaries of the filled groups are still connected with no cut points (though there are local cut points before and after filling).

We do not have an example where all the hypotheses of Theorem \ref{t:maintheorem} hold except for the fillings being co-slender and the conclusion does not hold, but we use the assumption that the fillings are co-slender in various places in our proofs.  We do not know if this assumption is necessary.

\begin{question}
Is the assumption that the fillings be co-slender required in Theorems \ref{t:easy maintheorem} and \ref{t:maintheorem}?
\end{question}

\begin{question}
  Does Theorem \ref{t:easy maintheorem} hold for \emph{all} sufficiently long fillings, without assuming $\MultiEnded$--finite and co-slender?  
\end{question}

\subsection{Outline}
In Section \ref{sec:prelim}, we recall the relevant definitions about Dehn filling, peripheral splittings.  In Section \ref{sec:boundsplit}, we recall some results of Bowditch about elementary splittings, and begin to study how they behave under filling.

Section \ref{sec:torsion} contains a result which may be of independent interest, on finite subgroups of relatively hyperbolic groups obtained via Dehn filling.

In Section \ref{s:rtrees} we prove Theorem \ref{t:easy maintheorem} by contradiction: If arbitrarily long fillings of $(G,\mc{P})$ have elementary splittings, the Bass-Serre trees of those splittings limit to an $\bR$--tree, from which we deduce the existence of an elementary splitting of $(G,\mc{P})$.  This also gives the setup for the proof of Theorem \ref{t:maintheorem}, which we then prove using the Rips theory of groups acting on trees (see \cite{Guirardel:Rtrees}) and a version of Sela's ``shortening argument".

In Section \ref{s:fuchsian}, we give a result about Fuchsian fillings of a relatively hyperbolic group.

Finally in Section \ref{s:boundaries} we relate our results to the structure of the Bowditch boundary.

\subsection{Acknowledgments}  Thanks to John MacKay for pointing out a typo, and to the referee(s) for useful comments and corrections.

\section{Preliminaries} \label{sec:prelim}

In this section we record the definitions and results that we need in the proof of the main results.

Suppose that $G$ is a finitely generated group and $\mc{P} = \{ P_1, \ldots , P_k \}$ is a collection of proper, finitely generated subgroups.  In \cite[$\S$3]{rhds} the construction of a {\em cusped space} is given, by gluing {\em combinatorial horoballs} onto the left cosets of the $P_i$ in a Cayley graph for $G$.  We denote such a cusped space by $X(G,\mc{P})$ (ignoring for the moment the generating sets of $G$ and of the $P_i$).  This space is a locally finite graph, on which $G$ acts properly.

\begin{definition}\label{def:RH} \cite{rhds}
The pair $(G,\mc{P})$ is {\em relatively hyperbolic} if the cusped space $X(G,\mc{P})$ is $\delta$-hyperbolic for some $\delta$.
\end{definition}
In case $(G,\mc{P})$ is relatively hyperbolic, the cusped space $X(G,\mc{P})$ has a Gromov boundary, which we refer to as the \emph{Bowditch Boundary}.  Since our peripheral structure $\mc{P}$ is not assumed to be minimal, this may be different from what some other authors call the Bowditch boundary.  For more on this see Section \ref{s:boundaries}.

In \cite[Theorem 3.25]{rhds} it is proved that Definition \ref{def:RH} agrees with the other notions of relatively hyperbolicity for finitely generated $G$.
(In \cite{Hr08} an extension to the case where $G$ is not finitely generated is given, and this definition still agrees with the standard ones in this setting.  We consider only the finitely generated case in this paper.)  

\begin{remark}
 If $(G,\mc{P})$ is relatively hyperbolic, the family $\mc{P}$ is \emph{almost malnormal}, in the sense that if $P,P' \in \mc{P}$ and $P^g \cap P'$ is infinite for some $g \in G$ then $P = P'$ and $g \in P$ (see \cite[Example 1, p.819]{farb:relhyp}).  
In particular, each $P \in \mc{P}$ is almost malnormal.  We use this property of $\mcP$ frequently, often without explicit mention.
\end{remark}

\begin{terminology}
In this paper, we consider relatively hyperbolic groups acting on trees.  Both relatively hyperbolic groups and groups acting on trees have `hyperbolic' elements, and these are different notions.

When we talk about an element being hyperbolic in the sense of relatively hyperbolic groups, we call it an {\em RH-hyperbolic element}.  These RH-hyperbolic elements are defined in the following subsection.

When talking about a group element being hyperbolic when acting on a tree, we mention the tree.
\end{terminology}

\subsection{Small, slender, and elementary subgroups}
The contents of this section are well known, but we have been unable to find some of the exact statements that we require in the literature.

\begin{definition}\label{def:elementary}
  For a relatively hyperbolic pair $(G,\mc{P})$, a subgroup $E<G$ is \emph{elementary} if it is either virtually cyclic or \emph{parabolic}, ie conjugate into some $P\in \mc{P}$.
\end{definition}

It is more convenient to think about the action of (subgroups of) $G$ on the cusped space $X(G,\mc{P})$ so we record the following.

\begin{lemma}
 Let $(G,\mc{P})$ be relatively hyperbolic.  A subgroup $H \le G$ is elementary if and only if $H$ preserves a finite set in $X(G,\mc{P}) \cup \partial(G,\mc{P})$.
\end{lemma}

  The action of a relatively hyperbolic group $(G,\mc{P})$ on its Bowditch boundary is \emph{geometrically finite} \cite[Proposition 6.15]{bowditch:relhyp}, meaning that every point is either a \emph{conical limit point} or a \emph{bounded parabolic point}.  (A converse to this statement was proved by Yaman \cite[Theorem 0.1]{yaman04}.)  The stabilizers of the parabolic points are exactly the conjugates of the elements of $\mc{P}$.
  
  An element of a relatively hyperbolic group $(G,\mc{P})$ is either finite order, infinite order parabolic (in which case it fixes a unique point in $\partial(G,\mc{P})$) or else {\em RH-hyperbolic}, in which case it fixes a pair of points in $\partial(G,\mc{P})$, preserves a quasi-geodesic axis between these points in the cusped space $X(G,\mc{P})$, and acts via {\em North-South} dynamics on $\partial (G,\mc{P})$.\footnote{Note that some finite order elements of $(G,\mc{P})$ may be parabolic, and others may be non-parabolic.}
 
\begin{definition}
 Let $(G,\mc{P})$ be relatively hyperbolic, and $\partial(G,\mc{P})$ the Bowditch boundary.   Suppose that $H \le G$ is infinite.  The {\em limit set} of $H$, denoted $\Lambda H$, is the minimal closed nonempty $H$-invariant subset of $\partial (G,\mc{P})$.
\end{definition}

\begin{definition}
Let $(G,\mc{P})$ be relatively hyperbolic.  A pair of RH-hyperbolic elements $g,h \in G$ are {\em independent} if their fixed sets in $\partial(G,\mc{P})$ are disjoint.
\end{definition}

The following theorem is due to Tukia.
\begin{theorem} \label{t:Tukia non-elem} \cite[Theorem 2U]{Tukia94}
Suppose that $H \le G$ has the property that $|\Lambda H| > 2$.  Then $H$ is non-elementary.  Every non-elementary subgroup contains a nonabelian free subgroup generated by two independent RH-hyperbolic elements. 
\end{theorem}

\begin{proposition} \label{p:elementary}
 Let $(G,\mc{P})$ be relatively hyperbolic, and let $H \le G$.  Then $H$ is elementary if and only if exactly one of the following occurs:
\begin{enumerate}
 \item $H$ is finite;
 \item $H$ is parabolic (conjugate into some element of $\mc{P})$ and $|\Lambda H| = 1$; or
 \item\label{virtualZ} $H$ is virtually infinite cyclic, and contains an RH hyperbolic element. In this case $|\Lambda H| = 2$.
\end{enumerate}
\end{proposition}
\begin{proof}
It is straightforward to see that each of the three types listed are elementary, which proves one direction of the theorem.

 Conversely, suppose that $H$ is infinite and elementary.  The action of $H$ on $X(G,\mc{P})$ is properly discontinuous, so the finite set preserved by $H$ must be contained in $\partial(G,\mc{P})$.  
 Since $\Lambda H$ is the minimal closed $H$-invariant set, we see that $\Lambda H$ is finite.
 
 If $\Lambda H$ contains more than $2$ points, then $H$ contains a nonabelian free group and $\Lambda H$ is infinite, by \cite[Theorem 2U]{Tukia94}.  Therefore, we are concerned with the cases where $|\Lambda H|$ has size $1$ or $2$.
 
 Suppose first that $\Lambda H = \{ p \}$.  If $p$ is a bounded parabolic point, then $H$ is parabolic, as required.
 
 Suppose then that $\Lambda H = \{ p \}$ and that $p$ is a conical limit point. We know that elements of $G$ are either finite order, infinite order parabolic or RH-hyperbolic (which are also infinite order).
 An infinite cyclic subgroup generated by an RH-hyperbolic element has a pair of limit points, so if $|\Lambda H| = 1$ then $H$ can contain no RH-hyperbolic elements.  On the other hand, if $H$ contains an infinite-order parabolic element, then it would have to fix $p$, which contradicts the assumption that $p$ is not a parabolic limit point.  Thus, every element of $H$ is finite-order, and hence elliptic. Therefore (since $p$ is $H$-invariant) each $h \in H$ coarsely preserves the level sets of a Busemann function $\beta_x$ based at $p$ (with the same implicit constant behind the word `coarse' for each element of $H$).  Let $D$ be the maximal difference $\beta_p(x) - \beta_p(h.x)$ over $x \in X$ and $h \in H$.  Let $x \in X$ be a base point, and $\gamma$ a geodesic ray from $x$ to $p$.  Since $p$ is not parabolic, $\gamma$ must intersect the Cayley graph of $G$ in $X$ infinitely many times.  Let $K$ be a bound on the size of the number of vertices in a ball of radius $D + 10\delta$ in $X$ which meets the Cayley graph of $G$.  We claim that (in case $p$ is $H$-invariant) $|H| \le K$, contradicting the hypothesis that $H$ is infinite.  Indeed, suppose that $H$ has distinct elements $h_1, \ldots , h_{K+1}$.  Then the geodesics $h_i . \gamma$ are all eventually within $2\delta$ of each other, which implies that at a point $y$ on $\gamma$ near the Cayley graph the points $h_i .y $ are all within $D + 2\delta$ of each other.  However, this means that there are $i$ and $j$ (with $i \ne j$) so that $h_i . y = h_j . y$.  Since $G$ acts freely on $X$, this is a contradiction.
 
 We are left with the case that $|\Lambda H | = 2$.  Suppose that $\Lambda H = \{ p,q \}$.  Suppose further that $H$ contains no RH-hyperbolic elements.
  In this case, given the basepoint $x \in X(G,\mc{P})$, there are elliptic or parabolic elements $h_1, h_2 \in H$ with the property that $d(h_1.x,x)$ and $d(h_2.x,x)$ are as large as we like, but the Gromov product $(h_1.x,h_2.x)_x$ is bounded.  In particular, by taking $h_1.x$ approximating $p$ and $h_2.x$ approximating $q$, we can ensure that 
\[ \min\{d(h_1.x,x),d(h_2.x,x)\}\geq 100 (h_1.x,h_2.x)_x+100\delta. \]
A standard argument (see, eg \cite[Chapitre 9, Lemme 2.3]{cdp}) shows that $h_1h_2$ is an RH-hyperbolic element.  This contradicts the assumption that $H$ has no RH-hyperbolic elements.

Therefore, in case $|\Lambda H| = 2$, we see that $H$ contains an RH-hyperbolic element.
Since $H$ preserves a pair of points in $\partial(G,\mcP)$, it is virtually cyclic.
(For example, this can be seen by noting that the action of $G$ on the cusped space is proper and then applying the classification of isometries of $\delta$-hyperbolic spaces \cite[Chapitre 9]{cdp}).\end{proof}

We immediately deduce the following:
\begin{lemma}\label{lem:elementary is fg}
 If $(G,\mc{P})$ is relatively hyperbolic and each element of $\mc{P}$ is slender then every elementary subgroup of $G$ is slender, and in particular finitely generated.
\end{lemma}
\begin{lemma}\label{lem:small is fg}
  Let $(G,\mcP)$ be relatively hyperbolic, where the elements of $\mcP$ are slender.  Then every small subgroup of $G$ is finitely generated.
\end{lemma}
\begin{proof}
  Let $H<G$ be small.  If $H$ is elementary then $H$ is finitely generated by Lemma \ref{lem:elementary is fg}.
  
  However, if $H$ is non-elementary then it contains a nonabelian free group by Theorem \ref{t:Tukia non-elem}, and so it cannot be small.
\end{proof}

Note that a slender group is small but a small group may not be slender.  Slender groups have the following useful characterization due to Dunwoody and Sageev.
\begin{lemma} \cite[Lemma 1.1]{DunwoodySageev99}\label{lem:slender tree}
A group $H$ is slender if and only if for every subgroup $K \le H$, every action of $K$ on a tree either has a fixed point or has a (setwise) invariant axis.
\end{lemma}

A general principle of relatively hyperbolic groups is that extreme pathology is usually confined to the parabolic subgroups.  The following is a well-known example of that principle.
\begin{lemma}\label{lem:itpar}
  Let $(G,\mcP)$ be relatively hyperbolic.  If $H<G$ is infinite torsion, then $H$ is parabolic.
\end{lemma}
\begin{proof}
Since $H$ is torsion, it contains no nonabelian free subgroup, so $H$ must be elementary by Theorem \ref{t:Tukia non-elem}.  However, the only possibility from the list of elementary subgroups given by Proposition \ref{p:elementary} for an infinite torsion group is if $H$ is parabolic.
\end{proof}

The following lemma will be used to show certain extensions of elementary groups coming from the Rips machine are still elementary.
\begin{lemma} \label{l:no peripheral-by-cyclic}
Suppose that $(G,\mc{P})$ is relatively hyperbolic.  \\
 Any elementary-by-(virtually abelian) subgroup of $G$ is elementary.\\
 Any (infinite parabolic)-by-(virtually abelian) subgroup of $G$ is parabolic.
\end{lemma}
\begin{proof}
Let $H$ be elementary-by-(virtually abelian).  There is a short exact sequence
\[	1 \to K \to H \to A \to 1	,	\]
where $A$ is virtually abelian.  There are a number of cases.

Suppose first that $K$ is infinite parabolic.  If $K$ is contained in a maximal parabolic subgroup $P$, then since $P$ is almost malnormal, all of $H$ is contained in $P$. 
In particular $H$ is parabolic, hence elementary.

Suppose then that $K$ is virtually cyclic with an RH-hyperbolic element.  Then the limit set of $K$ has two elements.  Moreover, since $K$ is normal in $H$, the limit set of $K$ is $H$-invariant, and so the limit set of $H$ also has two elements and $H$ is virtually cyclic, by Proposition \ref{p:elementary}.

Suppose next that $K$ is finite.  If $A$ is infinite torsion, then so is $H$.  Any infinite torsion subgroup of a relatively hyperbolic group is parabolic by Lemma \ref{lem:itpar}, so we can suppose that $A$ has an infinite order element $a$.  Let $\tilde{a}\in H$ map to $a$.  The subgroup $H_0 = \langle K,\tilde{a}\rangle$ is virtually infinite cyclic.  
We argue differently, depending on whether $\tilde{a}$ is parabolic or loxodromic.  If parabolic, $H_0$ fixes a unique $p\in \partial(G,\mcP)$; if loxodromic, $H_0$ fixes two points $p,q$ in $\partial(G,\mcP)$.  (No element can exchange them because $H_0$ maps homomorphically onto $\bZ$.)

Suppose $\tilde{a}$ is parabolic, fixing $p\in \partial(G,\mcP)$.  Note that $p$ is the unique fixed point for $\tilde{a}$.  Let $h\in H$.  The commutators of $h$ and $\tilde{a}^n$ are in the finite subgroup $K$, so we have, for some $k\in K$ and some $0<i<j$, $[h,\tilde{a}^i]=[h,\tilde{a}^j]=k$.  Thus
\[ h p = h\tilde{a}^i p = k \tilde{a}^i h p\mbox{, and}\]
\[ h p = h\tilde{a}^j p = k \tilde{a}^j h p.\]
So both $k\tilde{a}^i$ and $k\tilde{a}^j$ fix $q = hp$.  It follows that $\tilde{a}^i q = k^{-1}q =\tilde{a}^j q$, so $\tilde{a}^{j-i}$ fixes $q$.  We deduce $q = p$ so $h$ fixes $p$ and lies in the same parabolic subgroup as $H_0$.

Finally, suppose that $\tilde{a}$ is loxodromic with fixed points $\{p,q\}\subseteq \partial(G,\mc{P})$.  If $h\in H$, we see that $h$ preserves the pair $\{p,q\}$ by an argument similar to that in the last paragraph.  Since $H$ preserves a pair of points in $\partial(G,\mcP)$, it is virtually cyclic, by Proposition \ref{p:elementary}.
\end{proof}

\subsection{Dehn filling of groups}\label{ss:Dehn}

\begin{definition} \label{def:Dehnfill}
Suppose that $G$ is a group and $\mcP = \{ P_1 ,\ldots , P_n \}$ is a collection of subgroups.  A {\em  filling} (sometimes \emph{Dehn filling}) of $(G,\mcP)$ is determined by a collection of subgroups $K_i \unlhd P_i$ and is given by the quotient map (called the \emph{filling map})
\[	\phi \co G \to G(K_1,\ldots , K_n )	,	\]
where the group $G(K_1,\ldots , K_n)$ denotes $G / K$ where $K$ is the normal closure in $G$ of $\cup K_i$, and $\phi$ is the natural quotient map.
The subgroups $K_1,\ldots,K_n$ are called the {\em filling kernels}.
\end{definition}

\begin{definition}
Suppose that $(G,\mcP)$ is as in Definition \ref{def:Dehnfill}.  We say that a property $\mc{S}$ of groups holds {\em for all sufficiently long fillings} of $(G,\mcP)$ if there is a finite set $\mc{B} \in G \smallsetminus \{ 1 \}$ so that for any choice of filling kernels $K_1, \ldots , K_n$ so that $K_i \cap \mc{B} = \emptyset$ the group $G(K_1,\ldots , K_n)$ satisfies $\mc{S}$.
\end{definition}

The following is the main result of relatively hyperbolic Dehn filling.

\begin{theorem} \label{t:RHDF} \cite{osin:peripheral}, cf. \cite{rhds}
Let $(G,\mcP)$ be relatively hyperbolic and let $\mc{F} \subset G$ be finite.  For sufficiently long fillings
\[	\phi \co G \to G(K_1, \ldots , K_n) 	\]
of $(G,\mcP)$ the following properties hold:
\begin{enumerate}
\item For each $i$, the natural map $ P_i/K_i \to G(K_1, \ldots , K_n)$ is injective;
\item The pair $(G(K_1,\ldots , K_n), \{ P_1/K_1, \ldots , P_n/K_n \})$ is relatively hyperbolic; and
\item $\phi$ is injective on $\mc{F}$.
\end{enumerate}
\end{theorem}
\begin{notation}
  We'll also sometimes indicate a filling map by \[\phi\co (G,\mcP) \to (\barG,\barP)\] to emphasize the peripheral structure.  Here $\barP$ is the collection $\{\phi(P)\mid P\in\mcP\}$.
\end{notation}

In fact, there is uniform control on the geometry of the cusped spaces of the quotients.  Once generating sets are fixed for $G$ and the $P_i$, we can use the images of these generating sets to build the cusped spaces for the quotients after filling.  The cusped spaces are then determined completely.  In particular, the cusped space for a quotient $(\bar{G},\bar{\mc{P}})$ is an isometrically embedded subgraph of the quotient of $X(G,\mc{P})$ by the action of the kernel of the filling map $K$.  (The only difference between the two graphs is that the quotient by the action of $K$ may contain doubled edges and self-loops not present in $X(\bar{G},\bar{\mc{P}})$.)
\begin{theorem}\label{thm:uniformdelta} \cite[Theorem A.43(1)]{agol:virtualhaken}
  Using the cusped spaces just described, if $(G,\mc{P})$ is relatively hyperbolic, then:
  \begin{enumerate}
  \item There is a $\delta>0$ so that for all sufficiently long fillings $(G,\mc{P})\to (\bar{G},\bar{\mc{P}})$, the cusped space of $(\bar{G},\bar{\mc{P}})$ is $\delta$--hyperbolic, and
  \item Fix any (finite) ball $B\subseteq X(G,\mc{P})$.  For all sufficiently long fillings, the map $X(G,\mc{P})\to X(\bar{G},\bar{\mc{P}})$ restricts to an embedding of $B$ whose image is a metric ball in $X(\bar{G},\bar{\mc{P}})$.
  \end{enumerate}
\end{theorem}

\begin{definition}
A sequence $\{ \eta_i \co G \to \bar{G}_i \}$ of fillings is {\em stably faithful} if $\eta_i$ is faithful on the ball of radius $i$ about $1$ in $G$.
\end{definition}

\begin{corollary} \label{cor:stably faithful}
If $\mc{Q}$ is a property of groups and it is {\em not} the case that all sufficiently long fillings of a relatively hyperbolic pair $(G,\mc{P})$ satisfy $\mc{Q}$ then there is a stably faithful sequence of fillings $\{ \eta_i \co G \to \bar{G}_i \}$ so that for each $i$ the group $\bar{G}_i$ does not satisfy $\mc{Q}$.
\end{corollary}

\begin{lemma} \label{lem:stably hyperbolic}
Suppose that $(G,\mc{P})$ is relatively hyperbolic and that $g \in G$ is an RH-hyperbolic element.  Then for sufficiently long fillings $G \twoheadrightarrow \bar{G}$ the image of $g$ is RH-hyperbolic in $(\bar{G},\bar{\mc{P}})$.
\end{lemma}
\begin{proof}
  Let $\delta$ be so that cusped spaces of sufficiently long fillings of $(G,\mc{P})$ are all $\delta$--hyperbolic.  The existence of such a $\delta$ is the first part of Theorem \ref{thm:uniformdelta}, and we assume that the cusped space for $(G,\mc{P})$ is $\delta$-hyperbolic also.

  Let $g$ be an RH-hyperbolic element of $G$.  Then for some $\lambda,\epsilon$, $g$ preserves a $(\lambda,\epsilon)$--quasigeodesic axis $\gamma$, moving every point on $\gamma$ a distance of $|g|$, the translation length of $g$.  By \cite[Proposition 3.1.4]{cdp} there are $k,\lambda',\epsilon'$ (depending only on $\delta,\lambda,\epsilon$) so that every $k$--local $(\lambda,\epsilon)$--quasigeodesic is a global $(\lambda',\epsilon')$--quasigeodesic.  By quasigeodesic stability \cite[Th\'eor\`eme 3.1.2]{cdp} any such quasigeodesic lies in an $R_0$--neighborhood of a geodesic with the same endpoints, where $R_0$ depends only on $\delta,\lambda',\epsilon'$.  

Let $R = 2k+|g|+2R_0$, and
let $B$ be an $R$--ball centered on some point of $\gamma$.  
For sufficiently long fillings $B$ embeds in the cusped space of the filling, by the second part of Theorem \ref{thm:uniformdelta}.
Let $\bar{\gamma}$ be the image of $\gamma$.  We claim that $\bar\gamma$ is a $k$--local $(\lambda,\epsilon)$--quasigeodesic.  Indeed, let $\sigma$ be a subsegment of $\bar\gamma$ of length $k$.  There is a $k$--neighborhood of $\sigma$ contained in the image of $g^iB$ for some $i$, so $\sigma$ is the isometric image of a subsegment of the $(\lambda,\epsilon)$--quasigeodesic $\gamma$.  

If $\bar\gamma$ were a loop, it
would be a $(\lambda',\epsilon')$--quasigeodesic loop, which by quasigeodesic stability would have length at most $R_0$.  Since $R>2 R_0$, this is not the case.

It follows that $\bar\gamma$ is an infinite quasigeodesic, preserved by the image of $g$, so the image of $g$ is RH-hyperbolic. 
\end{proof}

\begin{remark} 
 We sketch another, almost formal proof of Lemma \ref{lem:stably hyperbolic}:  Let $E(g)$ be the maximal elementary subgroup containing the RH-hyperbolic element $g$ and note that $(G,\mcP\cup \{E(g)\})$ is relatively hyperbolic \cite[Corollary 1.7]{osin:bg}.  Extend the filling of $(G,\mcP)$ by a trivial filling of $E(g)$.  Then use almost malnormality in the quotient.  
\end{remark}

\subsection{Peripheral splittings} \label{ss:Peripheral Splittings}

\begin{definition} [Peripheral splitting]
Suppose that $(G,\mcP)$ is a relatively hyperbolic group.  A {\em peripheral splitting} of $(G,\mcP)$ is a bipartite graph of groups with fundamental group $G$ where the vertex groups of one color are precisely the peripheral subgroups $\mcP$.
\end{definition}

\begin{proposition} \cite[Proposition 5.1]{bowditch:peripheral} \label{prop:edge group parabolic}
The relatively hyperbolic pair $(G,\mcP)$ has a nontrivial peripheral splitting if and only if it has a nontrivial splitting over a (not necessarily maximal) parabolic subgroup in which all peripheral subgroups are elliptic.
\end{proposition}

\subsection{Automorphisms}
The following notion is required in the proof of Theorem \ref{t:maintheorem}.

\begin{definition}
Suppose that $G$ is a group and $\mc{P}$ is a collection of subgroups.  Let $\mathrm{Inn}(G)$ be the set of inner automorphisms of $G$, and let
\[ \Aut_{\mc{P}}(G) = \left\{\phi\in \Aut(G)\mid \forall P\in \mc{P},\ \exists \psi\in \mathrm{Inn}(G)\mbox{ so }\phi|_P = \psi|_P
\right\}. \]
\end{definition}
Thus elements of $\Aut_{\mc{P}}(G)$ are those automorphisms which restrict to an inner automorphism on each $P \in \mc{P}$ (a different inner automorphism for each $P$, possibly).

The following is a key example of an automorphism in $\Aut_{\mc{P}}(G)$.

\begin{example} \label{Ex:AutPG}
Suppose that $(G,\mc{P})$ is relatively hyperbolic and that $C$ is a two-ended subgroup whose center contains an RH-hyperbolic element $c$.  Suppose that $G = A \ast_C B$ and that each element of $\mc{P}$ is elliptic in this splitting (i.e. conjugate into $A$ or into $B$).  Let $\tau_c \co G \to G$ be the \emph{Dehn twist about $c$}, defined as
$\tau_c(a) = a$ if $a \in A$ and $\tau_c(b) = b^c$ if $b \in B$.  Then $\tau_c \in \Aut_{\mc{P}}(G)$.
\end{example}

Since the kernel of a filling map is the normal closure in $G$ of subgroups of elements of $\mc{P}$, the following is clear.

\begin{lemma}\label{lem:retain stably faithful}
Suppose that $(G,\mc{P})$ is a group pair, and that $\phi \in \Aut_{\mc{P}}(G)$.  If
\[	\eta \co G \to G(N_1,\ldots,N_m)	\]
is a Dehn filling of $(G,\mc{P})$ then $\ker(\eta \circ \phi) = \ker(\eta)$.

In particular, if $\{ \eta_i \co G \to \bar{G}_i \}$ is a stably faithful sequence of fillings and $\{ \phi_i \}$ is any sequence from $\Aut_{\mc{P}}(G)$ then $\{ \eta_i \circ \phi_i \}$ is a stably faithful sequence of fillings.
\end{lemma}

\section{Elementary splittings}\label{sec:boundsplit}

 Let $(G,\mc{P})$ be relatively hyperbolic.  Recall that a subgroup $H \le G$ is {\em elementary} if it is either finite, conjugate into some element of $\mc{P}$ or two-ended.
In this section, we are interested in splittings of relatively hyperbolic groups over elementary subgroups.  We call such splittings {\em elementary splittings}.   We are also interested in elementary splittings of quotients (under Dehn fillings).

Bowditch \cite[Proposition 5.2]{bowditch:peripheral} proves that certain splittings over parabolic subgroups can be improved to splittings in which all parabolic subgroups act elliptically (See also \cite[Lemma 2.1]{sela:dio1}).  We need a slightly different statement, but the proof we give below is essentially Bowditch's.    
\begin{lemma}\label{lem:make peripheral elliptic}
  Let $(G,\mc{P})$ be relatively hyperbolic, and suppose that $G$ admits a splitting over a parabolic or abelian subgroup so that every multi-ended element of $\mc{P}$ is elliptic.  Then $G$ admits a splitting over a parabolic or abelian subgroup so that every element of $\mc{P}$ is elliptic.

  If the original splitting was over a parabolic subgroup, then so is the $\mcP$--elliptic splitting.
\end{lemma}
\begin{proof}
  Let $T$ be the Bass-Serre tree for a one-edge splitting of $G$ over a parabolic or abelian subgroup, so that every multi-ended element of $\mc{P}$ is elliptic.  If every $P\in \mc{P}$ is elliptic in this tree, then there is nothing to prove.  

  So assume that some (necessarily one-ended) $P_0\in \mc{P}$ fixes no vertex of $T$, and let $T_0$ be a minimal $P_0$--invariant tree.  Fix some edge $e$ of $T_0$, and let $H$ be the stabilizer of $e$ in $G$.  Since $P_0$ is one-ended, $H\cap P_0$ must be infinite.  Since $H$ is an edge stabilizer, it is either parabolic or abelian.  
  \begin{claim}
    $H<P_0$.
  \end{claim}
  \begin{proof}
    Suppose that $H$ is parabolic.  Since $H\cap P_0$ is infinite, and $\mc{P}$ is almost malnormal, $H<P_0$.

    Suppose that $H$ is abelian.  Let $H_0 = H\cap P_0$, which we have observed is infinite.  Note that $H_0$ is contained in $hP_0h^{-1}\cap P_0$ for any $h\in H$.  But then $h\in P_0$, by almost malnormality of $P_0$.
  \end{proof}
  \begin{claim}
    Let $P \neq P_0$ be an element of $\mc{P}$.  Then $P$ fixes a point in $T$.
  \end{claim}
  \begin{proof}
    If $P$ didn't fix a point, $P$ would split over some $H' = P\cap H^g<P\cap P_0^g$.  By almost malnormality of $\mc{P}$, this group $H'$ would have to be finite, implying that $P$ was not one-ended.  But we are assuming multi-ended peripheral groups act elliptically.
  \end{proof}
  \begin{claim}
    $P_0$ is the setwise stabilizer of $T_0$.  In other words, $P_0 = \{g\in G\mid gT_0 = T_0\}$.  Moreover, if $g \not\in P_0$ then $| gT_0 \cap T_0 | \le 1$.
  \end{claim}
  \begin{proof}
    Let $g\in G$, and suppose that $e$ is an edge in $T_0 \cap gT_0$.  We proved above that if $H$ is the stabilizer of $e$ then $H < P_0$.  We want to show $g\in P_0$.  The stabilizer of $g^{-1}e$ is $g^{-1}H g$.  Some edge of the graph of groups $P_0\backslash T_0$ is labeled by a $P_0$-conjugate of the edge stabilizer $g^{-1}H g\cap P_0$.  Since $T_0$ is minimal, the corresponding one-edge splitting is nontrivial.  Since $P_0$ is one-ended, $g^{-1}H g\cap P_0$ is infinite, and so  $g\in P_0$ by almost malnormality of $P_0$.
  \end{proof}
  Now every edge in $T$ is in exactly one translate of $T_0$, so we have a partition of the edges of $T$ to which we can apply the construction from \cite[Lemma 3.5]{bowditch:peripheral}.  Namely, we let $S$ be a bipartite tree with red vertices equal to the set $\mc{A}$ of translates of $T_0$, and black vertices equal to the vertex set $V(T)$ of $T$.  Connect $v\in V(T)$ to $gT_0\in \mc{A}$ if and only if $v\in gT_0$.  Edge stabilizers are conjugate into $P_0$, so they are parabolic.  
  
  Finally, we explain why all elements $P \in \mc{P}$ fix points in $S$.  It is clear that $P_0$ fixes the vertex of $S$ corresponding to $T_0$.  On the other hand, if $P \in \mc{P}$ is not $P_0$ then it fixes a vertex of $T$, which is still a vertex of $S$.
\end{proof}

We next want to prove a result (Lemma \ref{lem:2csplit} below) saying that we can improve splittings of fillings to $(2,C)$--acylindrical splittings.  We'll need a result about uniformity of almost malnormality in fillings.  
\begin{definition}
  A collection of subgroups $\mcP$ of $G$ is \emph{$C$--almost malnormal} if there is a constant $C$ so that
 \[\#(P_1\cap g P_2 g^{-1})>C\mbox{ for }g\in G\mbox{ and } P_1,P_2\in \mcP \]
implies $P_2 = P_1$ and $g\in P_1$.
\end{definition}
The following is well-known.
\begin{lemma}\label{lem:CAM}
  Suppose $(G,\mcP)$ is relatively hyperbolic.  Then $\mcP$ is $C$--almost malnormal for some $C$.
\end{lemma}
\begin{proof}
  Let $P_1$ and $P_2$ be distinct conjugates of elements of $\mc{P}$.  Let $F = P_1\cap P_2$. 
  We use the action on the (combinatorial) cusped space $X(G,\mcP)$ with respect to some fixed generating set for $G$.  Since $(G,\mcP)$ is relatively hyperbolic, this space is $\delta$--hyperbolic for some $\delta>0$.  
  For each $i$, the subgroup $P_i$ preserves a combinatorial horoball $H_i$ whose center $e_i\in \partial X$ is the parabolic fixed point of $P_i$.  Choose a biinfinite geodesic $\gamma$ from $e_1$ to $e_2$ which is vertical inside $H_1$ and $H_2$.  
  We parametrize $\gamma$ so that $\gamma(0)$ is in the frontier of $H_1$; in particular $\gamma(0)$ is in the Cayley graph of $G$.

  For $f\in F\setminus \{1\}$ the geodesics $\gamma$ and $f\gamma$ form an ideal bigon.  Moreover, deep inside $H_1$ and $H_2$, the geodesics pass through vertices which are distance $1$ from one another.  In particular, we may choose large subsegments $\gamma([-N,R])$ and $f\gamma([-N,R'])$ whose endpoints are distance $1$ from one another.  From this it is easy to deduce $d(\gamma(0),f\gamma(0))\leq 2\delta+1$.    Thus $F$ acts freely on a subset of the Cayley graph of $G$ of diameter at most $2\delta+1$.  We can bound the cardinality of this set in terms of $\delta$ and the size of the generating set for $G$.
\end{proof}

The following is proved by a straightforward adaptation of the methods for analyzing height discussed in the Appendix to \cite{agol:virtualhaken}.  
\begin{proposition}\cite{gm:qcdf}\label{prop:Calmostmalnormal}
  If $(G,\mc{P})$ is relatively hyperbolic, and $\mcP$ is $C$--almost malnormal, then for all sufficiently long fillings $(\barG,\barP)$ of $(G,\mcP)$, the collection $\barP$ is $C$--almost malnormal.
\end{proposition}

\begin{definition}
  An action $G\acts T$ on a tree is \emph{$(k,C)$--acylindrical} if the stabilizer of any segment of length at least $k+1$ has cardinality at most $C$.  
\end{definition}
The proof of the following result is similar to that of Lemma \ref{lem:make peripheral elliptic}.
\begin{lemma}\label{lem:2csplit}
  Suppose that $(G,\mc{P})$ is relatively hyperbolic, where $\mcP$ is a $C$--almost malnormal collection of slender subgroups.  If $G$ admits a nontrivial splitting over a parabolic group, then $G$ admits a nontrivial $(2,C)$--acylindrical splitting over a parabolic group.
\end{lemma}
\begin{proof}
   Let $G\acts T$ be the Bass-Serre tree for a one-edge splitting of $G$ over a parabolic subgroup $E$.  We may suppose that $E<P$ for some $P\in \mcP$.  If $\#E\leq C$, there is nothing to prove, so we suppose that $\#E>C$.
  There are two cases, depending on whether or not $P$ acts elliptically on $T$.

  Suppose first that $P$ fixes a point of $T$.  If $P$ fixes an edge, then in fact $P=E$, and each edge stabilizer is equal to $P^g$ for some $g\in G$.  Since $\mcP$ is $C$--almost malnormal, any segment of length at least $2$ has stabilizer of size at most $C$, so $G\acts T$ is actually $(1,C)$--acylindrical.

  If $P$ fixes a vertex, but no edge of $T$, and $\sigma$ is a segment in $T$ of length at least $3$, then the stabilizer of $\sigma$ is again contained in the intersection of a pair of conjugates of $P$, and so has size bounded above by $C$.  Thus $G\acts T$ is $(2,C)$--acylindrical.

  Suppose now $P$ is not elliptic.  Since $P$ is slender, Lemma \ref{lem:slender tree} implies that $P$  preserves some line $l_P\subseteq T$ on which $P$ acts either by translations or as an infinite dihedral group.  In either case every edge in $l_P$ has the same stabilizer, namely $E$.

We claim that for any $g \not\in P$ we have $|g.l_P \cap l_P | \le 1$.  If not, there is an edge $e$ in this intersection. 
Then the stabilizer of $e$ is contained in the intersection of two distinct conjugates of $P$, which has size at most $C$.  This contradicts the assumption that $\#E>C$.

We can now form a simplicial $G$-tree $\hat{T}$ as follows:  The vertices are the $G$-translates of $l_P$ along with the vertices of $T$ and we join $g.l_P$ and $v$ when $v \in g.l_P$.  It is easy to see that this is a tree (see \cite[Lemma 3.5]{bowditch:peripheral}) upon which $G$ acts.  The edge stabilizers are subgroups of $P$ which are either $E$ or else have $E$ as an index $2$ subgroup (depending on whether $P$ acts by translations or dihedrally on $l_P$).  In particular they are parabolic.

Moreover, if we take three consecutive edges in $\hat{T}$ and an $\gamma$ element of $G$ which stabilizes them, then $\gamma$ stabilizes two different lines $g. l_P$ and $h.l_P$.  This implies that 
$\gamma \in P^g \cap P^h$ which has size at most $C$.  Therefore the $G$-action on $\hat{T}$ has parabolic edge stabilizers and is $(2,C)$-acylindrical.
\end{proof}

\begin{proposition} \label{prop:acylindrical}
Suppose that $(G,\mc{P})$ is a relatively hyperbolic group, so that $\mcP$ is $C$--almost malnormal.  For all sufficiently long co-slender fillings $(G,\mc{P}) \to (\bar{G},\barP)$, if $\bar{G}$ admits a nontrivial splitting over a parabolic group then $\bar{G}$ admits a nontrivial $(2,C)$-acylindrical splitting over a parabolic group.
\end{proposition}
\begin{proof}
  We suppose that the filling $(\barG,\barP)$ is long enough to apply Proposition \ref{prop:Calmostmalnormal}, and then apply Lemma \ref{lem:2csplit}.
\end{proof}

\section{Torsion in fillings}\label{sec:torsion}

The main result of this section is the following, which should be of independent interest.  
\footnote{A similar result about hyperbolically embedded subgroups
(but about only elements of finite order rather than finite subgroups) is proved in \cite[Theorem 7.19]{DGO}.}

\begin{theorem} \label{thm:torsion in fillings}
Suppose that $(G,\mc{P})$ is a relatively hyperbolic pair.  For all sufficiently long fillings
\[	\eta \co (G,\mc{P}) \to (\bar{G},\bar{\mc{P}})	\]
any finite subgroup of $\bar{G}$ is either
\begin{enumerate}
\item Conjugate into some element of $\bar{\mc{P}}$; or
\item The isomorphic image (under $\eta$) of some finite subgroup of $G$.
\end{enumerate}
\end{theorem}
We'll need the following lemma, which is used again in Section \ref{ss:controlsplit}.
\begin{lemma}\label{l:quasicenter}
  Let $F$ be a finite group acting on a $\delta$--hyperbolic space $X$.  Then $F$ has an orbit of diameter at most $4\delta+2$.
\end{lemma}
\begin{proof}
  Choose some $x_0\in X$.  Then $Fx_0$ is a nonempty bounded set, which has some radius $r$ (the radius is the infimum of those $R$ so that $Fx_0$ is contained in an $R$--ball about some point).  
  An \emph{$\epsilon$--quasicenter} for $Fx_0$ is a point $c$ so that $Fx_0\subseteq \{x\mid d(x,c)\leq r+\epsilon\}$.  Setting $\epsilon=1$, it is clear that there is at least one $1$--quasicenter $c$ for $Fx_0$, and that $fc$ is a $1$--quasicenter for $Fx_0$, for any $f\in F$.
  By \cite[Lemma III.$\Gamma$.3.3]{bridhaef:book}, the set of $1$--quasicenters has diameter at most $4\delta+2$.
\end{proof}

\begin{proof}[Proof of Theorem \ref{thm:torsion in fillings}]
By Theorem \ref{thm:uniformdelta},
there is a constant $\delta$ so that (i) The cusped space for $(G,\mc{P})$ is $\delta$-hyperbolic; and (ii) For all sufficiently long fillings the cusped space of the quotient is $\delta$-hyperbolic.  We may assume that all fillings we consider satisfy this condition, and we fix such a $\delta\geq 1$.  Let $X$ be the cusped space for $(G,\mc{P})$ and for a filling $(\bar{G},\bar{\mc{P}})$ we let the associated cusped space be $\bar{X}$.

Now take a filling which is long in the sense of the above paragraph and which also induces a bijection between the ball of radius $40\delta$ about $1$ in $X$ and the ball of radius $40\delta$ about $1$ in $\bar{X}$.

Suppose that $Q \le \bar{G}$ is a finite subgroup.
Lemma \ref{l:quasicenter} implies that there is a $Q$--orbit $B \subseteq \bar{X}$ of diameter at most $4\delta + 2$.

{\bf Case 1:}  Suppose that $B$ does not intersect the $(4\delta+2)$--neighborhood of the Cayley graph of $\bar{G}$ in $\bar{X}$.  Then since the diameter of $B$ is at most $4\delta + 2$, $B$ is entirely contained in a single horoball of $\bar{X}$, which implies that $Q$ is conjugate into some $\bar{P} \in \bar{\mc{P}}$.

{\bf Case 2:}  Now suppose that some $x \in B$ lies in the $(4\delta+2)$--neighborhood of the Cayley graph.  We recall from \cite{rhds} that vertices of the cusped space but not in the Cayley graph correspond to triples $(gP,g,n)$, where $P\in \bar{\mcP}$ and $n\in\bN =\{1,2,\ldots\}$.  We extend this labeling scheme to the Cayley graph by referring to $g\in \bar{G}$ using the (not uniquely defined) triple $(gP,g,0)$.
(The number $n$ is the distance from the Cayley graph.) A single combinatorial horoball is spanned by the vertices of the form $(hP,g,n)$ for $hP$ fixed.  

By assumption our $x = (gP,g,k)$ with $k \le 4\delta+2$.  Since the action of $\bar{G}$ on $\bar{X}$ is depth-preserving, we may assume that all other elements of $B$ have the form $(g'P,g',k)$ for the same $k$.  Consider the set
\[	B_0 = \{ h \mid (hP,h,k) \in B	\}		\]
(where we consider an element $g$ of $\bar{G}$ to be contained in $\bar{X}$ via the embedding $g \mapsto (gP,g,0)$). Then $Q \cdot B_0 = B_0$ and the diameter of $B_0$ is less than $20\delta$.

Moreover, $Q^g$ stabilizes $g^{-1}B_0$, and $1 = g^{-1} g \in g^{-1}B_0$.  This means that $Q^g = Q^g \cdot 1 \subset g^{-1}B_0$, which puts the subgroup $Q^g$ in the ball of radius $20\delta$ about $1$ in $\bar{X}$.

The filling induces a bijection between the $40\delta$ balls about the identity in $X$ and $\bar{X}$.  Let $\hat{Q}$ be the preimage of $Q^g$ under this bijection.  We claim that $\hat{Q}$ is a finite subgroup.  Indeed, for any $h_1$, $h_2\in Q^{g}$, let $h_3 = h_1h_2$.  For $i\in \{1,2,3\}$, let $\tilde{h}_i$ be the unique element of $\hat{Q}$ projecting to $h_i$.  Note that $\tilde{h}_1\tilde{h}_2$ lies in the $40\delta$ ball about $1$, and projects to $h_3\in Q^g$.  Since $\tilde{h}_3$ also projects to $h_3$, and the projection is injective on the ball of radius $40\delta$, we have $\tilde{h}_1\tilde{h}_2 = \tilde{h}_3$.  Since $h_1$ and $h_2$ were arbitrary, $\hat{Q}$ is a subgroup.

Letting $h$ be any element of $G$ which maps to $g$ under the filling, we see that $\hat{Q}^{h^{-1}}$ maps isomorphically to $Q$ under the filling, as required.
\end{proof}

\subsection{Controlling splittings over finite and two-ended subgroups}\label{ss:controlsplit}

\begin{definition}
  For $(G,\mcP)$ relatively hyperbolic, let $\mc{F}$ be the set of subgroups of $G$ which are either finite non-parabolic, or contained in the intersection of two distinct maximal parabolic subgroups.  Define $C(G,\mcP) = \max\{\#F\mid F\in\mc{F}\}$.
\end{definition}
\begin{lemma}
  For $(G,\mcP)$ relatively hyperbolic, $C(G,\mcP)<\infty$.  
\end{lemma}
\begin{proof}
  Let $F<G$ be finite and non-parabolic.  By Lemma \ref{l:quasicenter}, there is an $F$--orbit $B$ of diameter at most $4\delta+2$ in the cusped space.  Since $F$ is non-parabolic, $B$ is not contained entirely in a single horoball.  In particular $B$ contains a point within $4\delta+2$ of the Cayley graph, and so its size can be bounded above in terms of $\delta$ and the valence of the Cayley graph.  Since $\#F=\#B$, the size of such a group is bounded uniformly.

  If $F$ is in the intersection of two parabolics, the size of $F$ is bounded by Lemma \ref{lem:CAM}.
\end{proof}

We note the following corollary of Theorem \ref{thm:torsion in fillings} and Proposition \ref{prop:Calmostmalnormal}.
\begin{corollary}\label{cor:uniformC}
  Let $(G,\mcP)$ be relatively hyperbolic.  For all sufficiently long fillings $(\barG,\barP)$, we have $C(\barG,\barP)\leq C(G,\mcP)$.
\end{corollary}
In particular there is a bound on the size of non-parabolic finite subgroups of sufficiently long fillings, which we use in the following.  
\begin{corollary}\label{cor:finitesplittings}
  Suppose that $(G,\mcP)$ is relatively hyperbolic.  For all sufficiently long co-slender fillings $(G,\mcP)\to (\barG,\barP)$, if $\barG$ admits a nontrivial splitting over a finite group, then $\barG$ admits a nontrivial $(2,C)$--acylindrical splitting over a finite or parabolic group, where $C  = C(G,\mcP)$.
\end{corollary}
\begin{proof}
  Let $(G,\mcP)\to (\barG,\barP)$ be long enough that both Proposition \ref{prop:acylindrical} and Corollary \ref{cor:uniformC} apply, and suppose that $\barG$ splits nontrivially over a finite group $F$.  If $F$ is non-parabolic, then Corollary \ref{cor:uniformC} implies that $|F|\leq C$, so the Bass-Serre tree corresponding to the splitting over $F$ is already $(0,C)$--acylindrical.

If $F$ is parabolic, then \ref{prop:acylindrical} implies there is a $(2,C)$--acylindrical splitting over some parabolic group.  
\end{proof}

We next examine two-ended non-parabolic subgroups.
\begin{lemma}\label{lem:finite in twoended}
  Let $(G,\mcP)$ be relatively hyperbolic, let $E<G$ be two-ended and non-parabolic, and let $F<E$ be finite.  Then $\#F\leq 2C(G,\mcP)$.
\end{lemma}
\begin{proof}
    The two-ended group $E$ preserves a pair of points $E^{\pm \infty}$ in $\partial (G,\mcP)$.
  Let $E_0<E$ be the subgroup which fixes these points.  (If $E$ maps onto an infinite dihedral group $E_0$ is index two; otherwise $E=E_0$.)  Let $F_0 = F\cap E_0$, and let $\alpha\in E$ be an infinite order element which centralizes $F_0$.  If $F_0$ is non-parabolic, then we know $|F_0| \le C(G,\mc{P})$ by definition.  Suppose then that $F_0$ is parabolic, and contained in $gPg^{-1}$, for $P \in \mc{P}$.    Let $p$ be the parabolic fixed point for $gPg^{-1}$.
  Then for any $f\in F_0$ we have $f\alpha p = \alpha f p = \alpha p$, so that $f$ fixes both $p$ and $\alpha p$.  Note that we know $\alpha$ is not parabolic, since $E$ is not parabolic, so we know that $\alpha p \ne p$.  Thus $F_0$ is in the intersection of two distinct maximal parabolic subgroups and $\#F_0\leq C(G,\mcP)$.  The result follows.
\end{proof}

If $E$ is a maximal two-ended subgroup of $G$, where $(G,\mcP)$ is relatively hyperbolic, and $E$ is not already parabolic, then $(G,\mcP\sqcup\{E\})$ is also relatively hyperbolic \cite[Corollary 1.7]{osin:bg}.  The next lemma tells us how $C(G,\mcP)$ changes after augmenting the peripheral structure in this way.
\begin{lemma}\label{lem:augment peripheral}
  Suppose that $(G,\mcP)$ is relatively hyperbolic and $E<G$ is a maximal two-ended subgroup of $G$ which is not parabolic.  Then $C(G,\mcP\sqcup\{E\})\leq 2C(G,\mcP)$.  
\end{lemma}
\begin{proof}
  Let $\mcP' = \mcP\sqcup \{E\}$.
  Let $F$ be a finite subgroup of $G$ which is not parabolic with respect to $\mcP'$.  Then $F$ is not parabolic with respect to $\mcP$ either, so $\#F\leq C(G,\mcP)$.  Likewise, if $F$ is in the intersection of two $\mcP$--parabolic subgroups, then $\#F\leq C(G,\mcP)$.  

  It remains to consider intersections of conjugates of $E$ with each other or with other parabolics.  But Lemma \ref{lem:finite in twoended} bounds the size of finite subgroups of $E$ by $2C(G,\mcP)$.
\end{proof}

\begin{lemma}\label{lem:2csplit two-ended}
  Let $(G,\mcP)$ be relatively hyperbolic, where each element of $\mcP$ is slender, and let $C = 2C(G,\mcP)$.  If $G$ admits a nontrivial splitting over a two-ended non-parabolic subgroup, then $G$ admits a nontrivial $(2,C)$--acylindrical splitting over an elementary subgroup.
\end{lemma}
\begin{proof}
  Suppose $G$ splits over a two-ended non-parabolic subgroup $E$, and let $\hat{E}$ be the maximal two-ended subgroup of $G$ containing $E$.  Let $\mcP' = \mcP\sqcup\{\hat{E}\}$
  As we remarked before Lemma \ref{lem:augment peripheral}, $(G,\mcP')$ is relatively hyperbolic.  By Lemma \ref{lem:augment peripheral}, $C(G,\mcP')\leq C = 2C(G,\mcP)$.  Lemma \ref{lem:2csplit} implies that $G$ admits a $(2,C)$--acylindrical splitting over a $\mcP'$--parabolic subgroup $S$.  If $S$ is not $\mcP$--parabolic, then it must be conjugate to a subgroup of $E$, so it is either finite or $2$--ended.
\end{proof}

\begin{proposition} \label{l:two-ended acylindrical}
 Suppose that $(G,\mc{P})$ is relatively hyperbolic, and let $C=2C(G,\mcP)$.  For all sufficiently long co-slender fillings $(G,\mc{P}) \to (\bar{G},\bar{\mc{P}})$, if $\bar{G}$ admits a nontrivial splitting over a two-ended non-parabolic subgroup then $\bar{G}$ admits a nontrivial $(2,C)$--acylindrical splitting over an elementary subgroup.
\end{proposition}
\begin{proof}
  By Corollary \ref{cor:uniformC}, $C(\barG,\barP) = C(G,\mcP)$ for sufficiently long fillings.  Apply Lemma \ref{lem:2csplit two-ended}.
\end{proof}

\section{Limiting actions on $\R$-trees} \label{s:rtrees}

In this section, we give the proofs of the main Theorems \ref{t:easy maintheorem} and \ref{t:maintheorem}.  In each case we assume that the theorem is false and investigate a sequence of longer and longer fillings contradicting the conclusion.

\begin{lemma}\label{lem:setup}
 Suppose that $(G,\mc{P})$ is relatively hyperbolic and a counterexample either to Theorem \ref{t:easy maintheorem} or to Theorem \ref{t:maintheorem}.
 
Let $C=2C(G,\mcP)$.  There is a stably faithful sequence of $\MultiEnded$--finite co-slender fillings $\eta_i \co (G,\mc{P}) \to (\bar{G}_i,\bar{\mc{P}}_i)$ so that for each $i$ the group $\bar{G}_i$ admits a nontrivial $(2,C)$-acylindrical elementary splitting.  In case $(G,\mcP)$ is a counterexample to Theorem \ref{t:maintheorem}, we may assume these splittings are over parabolic or finite subgroups.
\end{lemma}
\begin{proof}
 Suppose first that $(G,\mc{P})$ is a counterexample to Theorem \ref{t:easy maintheorem}.  Then there is a stably faithful sequence of $\MultiEnded$--finite co-slender fillings $(G,\mc{P}) \to (\bar{G}_i,\bar{\mc{P}}_i)$ so that each $\bar{G}_i$ admits a nontrivial elementary splitting.  By Corollary \ref{cor:finitesplittings}, Proposition \ref{prop:acylindrical}, and Proposition \ref{l:two-ended acylindrical} these splittings can be modified to $(2,C)$--acylindrical elementary splittings, for all sufficiently large $i$.

 Now suppose that $(G,\mc{P})$ is a counterexample to Theorem \ref{t:maintheorem}.  In this case, there is a stably faithful sequence of $\MultiEnded$--finite co-slender fillings $\eta_i \co (G,\mc{P}) \to (\bar{G}_i,\bar{\mc{P}}_i)$ so that each $\bar{G}_i$ admits a nontrivial splitting over either a finite or a parabolic subgroup.  By Corollary \ref{cor:finitesplittings} and Proposition \ref{prop:acylindrical}, these splittings can be modified to $(2,C)$--acylindrical splittings over finite or parabolic subgroups, for all sufficiently large $i$.
\end{proof}

We now suppose that we have a relatively hyperbolic pair $(G,\mc{P})$, and that we have a stably faithful sequence of $\MultiEnded$--finite co-slender fillings $\eta_i\co (G,\mc{P}) \to (\bar{G}_i,\bar{\mc{P}}_i)$ so that each $\bar{G}_i$ admits a nontrivial $(2,C)$-acylindrical elementary splitting.

Choose a finite generating set $\mc{A}$ for $G$.  Each $\bar{G}_i$ acts $(2,C)$--acy\-lin\-dri\-cally on the Bass-Serre tree $T_i$ of its splitting, and the map $\eta_i$ induces an action of $G$ on $T_i$.  Choose a base vertex $x_i \in T_i$ which is \emph{centrally located} in the sense that it minimizes the function $x\mapsto \max\limits_{a\in \mc{A}}d_T(x,ax)$ on $T_i$.  

Given any action $\lambda \co G \to \Isom(T)$ (for some tree $T$ with centrally located basepoint $x$) define a length (with respect to $\mc{A}$)
\[	\| \lambda \| = \max_{a \in \mc{A}} d_T(x,\lambda(a).x)	.	\]
Thus by choosing a $(2,C)$--acylindrical Bass-Serre tree $T_i$ and centrally located point $x_i\in T_i$ as above, we obtain a length.  Abusing notation slightly, we call this number $\|\eta_i\|$; we'll use it as a scaling factor.

\begin{definition}
  Let $(T,x)$ be a basepointed tree, and let $(G,\mc{P})$ be relatively hyperbolic, and generated by $\mc{A}$.
  An action $\eta\co G\to \mathrm{Isom}(T)$ is \emph{shortest} (with respect to $\mc{P}$ and $\mc{A}$) if, for every $\phi\in \Aut_{\mc{P}}(G)$, we have
  \begin{equation}\label{shortest}
    \|\eta\circ \phi\|\ge \|\eta\|
  \end{equation}
\end{definition}

We summarize the assumptions we'll make for the remainder of this section.
\begin{assumption}\label{assumption:shortest}
$(G,\mcP)$ is a counterexample to either Theorem \ref{t:easy maintheorem} or Theorem \ref{t:maintheorem}.  Let $C = 2C(G,\mcP)$. 
  \begin{enumerate}
  \item\label{a:small fg} Small subgroups of $G$ are finitely generated.  
\item\label{a:G 1-ended} $G$ is one-ended and admits no proper peripheral splitting.
  \item\label{a:sfaith}  $(\bar{G}_i,\bar\mcP_i)$ is a stably faithful sequence of $\MultiEnded$--finite co-slender fillings of $(G,\mc{P})$ exhibiting the fact that $(G,\mcP)$ is a counterexample; for each $i$, there is a $(2,C)$--acylindrical splitting of $\bar{G}_i$ as in the conclusion of Lemma \ref{lem:setup}.
  \item\label{a:short} For each $i$, the action $\eta_i$ on the associated Bass-Serre tree is shortest (with respect to $\mc{P}$ and $\mc{A}$).
  \end{enumerate}
\end{assumption}
Lemma \ref{lem:retain stably faithful} implies that we can arrange Assumption \ref{assumption:shortest}.\eqref{a:short} without disturbing Assumption \ref{assumption:shortest}.\eqref{a:sfaith}.

Assumption \eqref{a:short} is only required in the proof of Theorem \ref{t:maintheorem}, not in the proof of Theorem \ref{t:easy maintheorem}.  However, making this assumption in both cases causes no harm.

We let $D_i = \|\eta_i\|$, and note that $D_i>0$, or the tree $T_i$ would have a global fixed point.

We thus obtain a sequence of actions of $G$ on simplicial trees.  
In case the sequence $\{ D_i \}$ is bounded, we obtain a limiting action on a simplicial tree $T_\infty$ by noting that the Lyndon length functions on some subsequence converge to a $\bZ$--valued Lyndon length function.  This implies there is an invariant simplicial subtree by \cite{chiswell,alperinmoss}.
If on the other hand the sequence $\{ D_i \}$ is unbounded, then by rescaling the metrics by $\frac{1}{D_i}$ we obtain a limiting action of $G$ on an $\R$-tree $T_\infty$ (see \cite[Theorem 3.3]{bestvina:handbook}).  
We assume in either case that the $\R$-tree $T_\infty$ has no proper $G$-invariant subtree (by passing to such a subtree if necessary).  This minimal subtree contains the basepoint $x_\infty$ which is a limit of the basepoints $x_i\in T_i$.

\begin{lemma}
The action of $G$ on $T_\infty$ has no global fixed point.
\end{lemma}
\begin{proof}
Suppose first that the stretching factors $\{ D_i \}$ diverge.  Note that by construction the basepoint $x_\infty$ is not a global fixed point, since some generator in $\mc{A}$ moves $x_\infty$ distance $1$.  Because the inner automorphisms of $G$ are in $\Aut_{\mc{P}}(G)$, from a global fixed point we could conjugate to find a homomorphism which moves $x_i$ a shorter distance for sufficiently large $i$, contradicting Assumption \ref{assumption:shortest}.

In case the stretching factors don't diverge, the limiting $\R$-tree $T_\infty$ is simplicial, so we obtain a graph of groups decomposition of $G$ coming from a limit of actions on the Bass-Serre trees of the $\bar{G}_i$.  If there were a global fixed point for the $G$-action on $T_\infty$, this splitting would be trivial.  However, if the splitting of $G$ induced by $T_\infty$ is trivial it is easy to see that for sufficiently large $i$ the splitting of $\bar{G}_i$ is trivial, in contradiction to the choice of the $\bar{G}_i$.
\end{proof}

\begin{lemma} \label{lem:cyclic is elliptic}
Suppose that $P \in \mc{P}$ is multi-ended. Then $P$ acts elliptically on $T_\infty$.
\end{lemma}
\begin{proof}
For any multi-ended $P \in \mc{P}$, the $\MultiEnded$--finiteness assumption implies that the image of $P$ in each $\bar{G}_i$ is finite, and so $P$ acts elliptically on each $T_i$.  On the other hand, a generator for $P$ has some word length in $\mc{A}$, which means that this generator moves the basepoint $x_i$ a distance which is a bounded multiple of $D_i$. This implies that a fixed point for $P$ in $T_i$ is distance a bounded multiple of $D_i$ from $x_i$, which in turn implies that these fixed points persist in the limit, so $P$ fixes a point in $T_\infty$. 
\end{proof}

\begin{lemma}
The scaling factors $D_i$ are unbounded.
\end{lemma}
\begin{proof}
Suppose not.  Then the limiting tree $T_\infty$ is a (minimal) simplicial tree.  Let $E$ be an edge stabilizer for this tree.  Since $G$ is one-ended, $E$ is infinite.  We claim that $E$ cannot be parabolic.  Indeed, by Lemma \ref{lem:cyclic is elliptic}, any multi-ended peripheral subgroup of $\mcP$ acts elliptically on $T_\infty$.  So by Lemma \ref{lem:make peripheral elliptic}, if $E$ were parabolic there would be a nontrivial splitting of $G$ in which all elements of $\mcP$ were elliptic.  By Proposition \ref{prop:edge group parabolic}, $G$ would admit a proper peripheral splitting.

Since $E$ is infinite non-parabolic, it contains an RH-hyperbolic element $g$.  (Either it is non-elementary, in which case we can apply Theorem \ref{t:Tukia non-elem}, or it is elementary, and we are in case \eqref{virtualZ} of Proposition \ref{p:elementary}.)

By Lemma \ref{lem:stably hyperbolic}, the image of $g$ in $\bar{G}_i$ is RH-hyperbolic for large $i$.    This means that the elementary splitting of $\bar{G}_i$ giving rise to $T_i$ is not over a parabolic subgroup or a finite subgroup, and so must be over a two-ended non-parabolic subgroup.

Here the argument diverges depending on whether $G$ is a counterexample to Theorem \ref{t:easy maintheorem} or \ref{t:maintheorem}.

In case $G$ is a counterexample to Theorem \ref{t:maintheorem}, we have already reached a contradiction, since the $T_i$ are assumed to come from parabolic or finite splittings.

In case $G$ is a counterexample to Theorem \ref{t:easy maintheorem}, we conclude that $E$ is nonelementary.  In particular $E$ contains a nonabelian free subgroup.  The bound in Lemma \ref{lem:finite in twoended} implies that infinitely many edge groups $E_i < \bar{G}_i$ are isomorphic.  Note that the edge groups $E_i$ are all quotients of $E<G$.  
But since every
two-ended group satisfies a law, there is no stably faithful sequence of homomorphisms from a group containing a nonabelian free group to a fixed two-ended group.  Thus in this case we have also reached a contradiction.
\end{proof}

Since the scaling factors $D_i$ are unbounded, the limiting $\R$-tree $T_\infty$ may not be simplicial.  In order to obtain our contradiction to prove Theorem \ref{t:easy maintheorem}, we apply the Rips machine (in its version from \cite{Guirardel:Rtrees}) to find an elementary splitting of $G$.  For the contradiction for Theorem \ref{t:maintheorem}, we have to undertake a more refined analysis of the limiting $\R$-tree (still using the results from \cite{Guirardel:Rtrees}) and then use Sela's `Shortening Argument' to argue that for large $i$ the action $\eta_i$ is not shortest, contrary to Assumption \ref{assumption:shortest}.

We start by analyzing the arc stabilizers of the $G$-action on $T_\infty$.

\begin{lemma} \label{lem:arc stab structure}
Let $I$ be a nondegenerate arc in $T_\infty$.  The stabilizer of $I$ fits into a short exact sequence
\[ 1\to N\to \stab(I)\to A\to 1, \]
where $A$ is abelian, and $N$ is finite of order at most $C$.
\end{lemma}
\begin{proof}
This argument is very similar to the argument from the proof of \cite[Proposition 1.2(i), page 531]{sela:acyl}.
We show that the commutator subgroup of $\stab(I)$ has cardinality bounded above by $C$.

Indeed, let $\gamma_1,\ldots,\gamma_n$ be distinct elements of the commutator subgroup of $\stab(I)$.  
For large $i$ the set $\{\gamma_1,\ldots,\gamma_n\}$ embeds in $\bar{G}_i$, since the sequence is stably faithful.
For each $\gamma_j$ fix some expression of $\gamma_j$ as $\prod_{k=1}^{n_j} [\alpha_{j,k},\beta_{j,k}]$, a product of commutators of elements of $\stab(I)$; let $F$ be the set of $\alpha_{j,k},\beta_{j,k}$ which occur in one of these expressions.  

The endpoints $p$ and $q$ of $I$ are limits of sequences $\{p_i\}$, $\{q_i\}$ of vertices in the approximating Bass-Serre trees, so the distances $\frac{d(p_i,q_i)}{D_i}$ tend to $d(p,q)$.  The elements of $F$ all stabilize $I$,  so for sufficiently large $i$ we have 
$\max\{d(p_i,fp_i),d(q_i,fq_i)\}<\frac{D_i}{100}$ for all $f\in F$, and $D_i>1000$.  
It follows that there is a large subsegment of $[p_i,q_i]$ on which all the elements of $F$ act like translations (or act trivially).  
It follows that all the elements $\gamma_1,\ldots,\gamma_n$ lie in the stabilizer of this segment.  Since the action of $\bar{G}_i$ on $T_i$ is $(2,C)$--acylindrical, $n\leq C$, as required.
\end{proof}
Recall that a group is \emph{small} if it has no nonabelian free subgroup.  Finite-by-abelian groups are clearly small.
\begin{corollary}\label{cor:stab small}
  The action $G\acts T_\infty$ has small arc stabilizers.
\end{corollary}

\begin{lemma}\label{lem:tripod stabilizers}
Let $Y$ be a nontrivial tripod in $T_\infty$.  The stabilizer of $Y$ is finite of order at most $C$.
\end{lemma}
\begin{proof}
This argument is very similar to one in Sela \cite[Proposition 1.2(ii), page 531]{sela:acyl}, but modified slightly to deal with torsion.  Let $g_1,\ldots,g_n\subseteq \stab(Y)$, let $p,q,r$ be the leaves of $Y$, and let $z$ be its center.  There are sequences of points $p_i,q_i,r_i\in T_i$ approximating $p,q,r$, and it is not hard to see that if $z_i$ is the center of the tripod spanned by $\{p_i,q_i,r_i\}$, then the sequence $\{z_i\}$ tends to $z$.  For sufficiently large $i$, the distances $d(p_i,g_jp_i), d(q_i,g_j p_i), d(r_i,g_jr_i)$ are all much smaller than $d(p_i,z_i), d(q_i,z_i), d(r_i,z_i)$, and these latter three distances are large.  It follows that, for large $i$, all the $g_j$ fix a large tripod centered at $z_i\in T_i$; in particular they lie in the stabilizer of some segment of length at least three.

Since the splittings giving rise to the trees $T_i$ are $(2,C)$--acylindrical, and the sequence $\bar{G}_i$ is stably faithful, we have $n\leq C$ for sufficiently large $i$.
\end{proof}

An arc $I$ of a $G$--tree $T$ is \emph{unstable} if it contains a nondegenerate subarc $J$ with strictly larger stabilizer.
The following result follows easily from the last two lemmas, see \cite[Proposition 4.2]{RipsSela94}, and says that the action is ``almost superstable''.  

\begin{lemma} \label{lem:unstable arc}
The stabilizer in $G$ of an unstable arc in $T_\infty$ has order at most $C$.
\end{lemma}

The following is an immediate corollary of Theorem \ref{t:Tukia non-elem}.
\begin{lemma}\label{lem:small is elementary}
  Let $(G,\mcP)$ be relatively hyperbolic.  Then all small subgroups of $G$ are elementary.
\end{lemma}

We now complete the proof of Theorem \ref{t:easy maintheorem}, by appealing to a result of Guirardel.  Recall that Corollary \ref{cor:stab small} tells us that $T_\infty$ has small arc stabilizers.

\begin{theorem}(simplified version of \cite[Corollary 5.3]{Guirardel:Rtrees})\label{Guirardeltheorem:easy}
  Let $G$ be a finitely generated group for which any small subgroup is finitely generated, and suppose $G\acts T$ where $T$ is an $\bR$--tree with small arc stabilizers.  Then $G$ splits over a small subgroup.  
\end{theorem}

Recall that we are assuming (see Assumption \ref{assumption:shortest}.\eqref{a:small fg}) that small subgroups of $G$ are finitely generated, so this assumption in Theorem \ref{Guirardeltheorem:easy} is satisfied.

Recall the statement of Theorem \ref{t:easy maintheorem}.
\easymaintheorem*
\begin{proof}
  If $(G,\mcP)$ is a counterexample to this theorem, we have shown in this section how to build a fixed-point free minimal action of $G$ on an $\bR$--tree $T_\infty$ from a sequence of Bass-Serre trees of $(2,C)$--acylindrical splittings of a stably faithful sequence of fillings $G\to \bar{G}_i$.  Corollary \ref{cor:stab small} says that arc stabilizers for this $\bR$--tree are small.  Since small subgroups of $G$ are finitely generated, we can apply Theorem \ref{Guirardeltheorem:easy} to conclude that $G$ splits over a small, and hence elementary subgroup, which is a contradiction.
\end{proof}

In order to conclude the proof of Theorem \ref{t:maintheorem}, we need to get into more details of the Rips machine and then apply (a version of) Sela's `Shortening Argument'.

\subsection{Proof of Theorem \ref{t:maintheorem}}

Recall the statement:
\maintheorem*

The argument given above to prove Theorem \ref{t:easy maintheorem} produces an elementary splitting.  Such a splitting contradicts the hypotheses of Theorem \ref{t:easy maintheorem}, but in case the splitting is over a non-parabolic (and hence two-ended) subgroup, it does not contradict the hypotheses of Theorem \ref{t:maintheorem}.  To prove Theorem \ref{t:maintheorem} we must work harder to obtain a contradiction,
further analyzing the limiting action on the $\R$-tree $T_\infty$ and using this analysis to argue that
for large enough $i$ the action $\eta_i$ can be shortened using an automorphism in $\Aut_{\mcP}(G)$.  This will 
contradict the Standing Assumption \ref{assumption:shortest}.\eqref{a:short}.  We now give more details.

We'll use a result of Guirardel (Theorem \ref{thm:Guirardel} below) to show that the action of $G$ on $T_\infty$ can be decomposed into a `transverse covering':

\begin{definition} \cite[Definition 1.4]{Guirardel:Rtrees}
 A {\em transverse covering} of an $\R$-tree $T$ is a covering of $T$ by a family of subtrees $\mc{Y} = \left( Y_v \right)_{v \in V}$ such that
\begin{enumerate}
 \item[$\bullet$] every $Y_v$ is a closed subtree of $T$;
 \item[$\bullet$] every arc of $T$ is covered by finitely many subtrees of $\mc{Y}$;
 \item[$\bullet$] for $v_1 \ne v_2 \in V$, $Y_{v_1} \cap Y_{v_2}$ contains at most one point.
\end{enumerate}
When $T$ comes equipped with an isometric action of a group $G$, we always require the family $\mc{Y}$ to be $G$-invariant.
\end{definition}

Given a transverse covering $\mc{Y}$ of $T$, one can define a bipartite simplicial tree by taking one family of vertices $V_1(\mc{Y})$ to be the set of subtrees $Y \in \mc{Y}$ and $V_0(S)$ to be the set of points $x \in T$ which lie in at least two distinct subtrees from $\mc{Y}$.  Join an edge from $x \in V_0(S)$ to $Y \in V_1(S)$ when $x \in Y$.  This builds a simplicial tree $S$, called the {\em skeleton} of $\mc{Y}$.  If there is a group $G$ acting on $T$ and $\mc{Y}$ then $G$ naturally acts on the skeleton $S$.

The data of a $G$-equivariant transverse covering and skeleton $S$ give a {\em graph of actions} decomposition for the $G$-action on $T$.  

Two types of tree actions which do not decompose in this way are \emph{surface type} and \emph{axial}, which we now define.
\begin{definition} \label{def:surface tree}
  A $G$--tree $T$ is \emph{of surface type} if there is an epimorphism $G\to \pi_1\Sigma$, where $\Sigma$ is a hyperbolic $2$--orbifold (possibly with boundary) and $T$ is the $\bR$--tree dual to a filling measured lamination on $\Sigma$ with no closed leaves.  (Here $\pi_1\Sigma$ refers to the orbifold fundamental group of $\Sigma$.)
\end{definition}
\begin{remark} \label{rem:surface arc stabs}
Consider a surface-type tree $T$ with the group $\pi_1\Sigma$ acting as in Definition \ref{def:surface tree} above, and let $I$ be a nontrivial arc in $T$.  The stabilizer of $I$ in $\pi_1\Sigma$ is trivial.
\end{remark}

\begin{definition}\label{def:axial tree}
  A $G$--tree $T$ is \emph{axial} if $T\cong\bR$ and $G$ acts as a finitely generated indiscrete subgroup of $\Isom(\bR)$.
\end{definition}
See \cite{Guirardel:Rtrees} and \cite{Guirardel04} for more details.  The key result we need is the following:

\begin{theorem}\cite[Corollary 5.3]{Guirardel:Rtrees} \label{thm:Guirardel}
  Suppose $G$ is a finitely generated group for which any small subgroup is finitely generated, and suppose $G\acts T$ where $T$ is an $\bR$--tree with small arc stabilizers.  Suppose further that $G$ does not split over any tripod stabilizer or over the stabilizer of an unstable arc.  Then $T$ has a transverse covering giving a graph of actions in which every vertex action is either
\begin{enumerate}
\item simplicial;
\item of surface type; or
\item axial.
\end{enumerate}
\end{theorem}

\begin{corollary} \label{cor:tree decomp}
  The tree $G\acts T_\infty$ constructed above has a decomposition as a graph of actions where each vertex action is either simplicial, of surface type, or axial.  Moreover elements of $G$ which act elliptically on $T_\infty$ also act elliptically on the skeleton, $S$, of this graph of actions.
\end{corollary}
\begin{proof}
  Corollary \ref{cor:stab small} says that the arc stabilizers of $G\acts T_\infty$ are small.  Lemmas \ref{lem:tripod stabilizers} and \ref{lem:unstable arc} imply that stabilizers of tripods and of unstable arcs are always finite.  By Assumption \ref{assumption:shortest}.\eqref{a:G 1-ended}, $G$ is one-ended, and in particular $G$ doesn't split over one of these stabilizers.  Theorem \ref{thm:Guirardel} therefore gives us a graph of actions as specified.

By \cite[Lemma 1.15]{Guirardel:Rtrees}, elements of $G$ which act elliptically on $T_\infty$ also act elliptically on $S$.
\end{proof}
For the remainder of the section, we fix this graph of actions decomposition just obtained.  The vertex trees $Y_v$ are subtrees of $T_\infty$ and we also refer to them as \emph{components} of the action $G\acts T_\infty$.

\begin{lemma}
  $G\acts T_\infty$ has no axial components.
\end{lemma}
\begin{proof}
Suppose that $G \acts T_\infty$ has an axial component, $T_v$.  Then the associated vertex group $G_v$ admits a short exact sequence
\[	1 \to K \to G_v \to A \to 1	,	\]
where $K$ acts trivially on $T_v$ and $A$ is an indiscrete subgroup of $\mathrm{Isom}(\R)$.  Since $T_v$ contains an arc, $K$ is small by Corollary \ref{cor:stab small} and hence elementary by Lemma \ref{lem:small is elementary}.  By Lemma \ref{l:no peripheral-by-cyclic} this implies that $G_v$ is elementary.  However, $A$ is an indiscrete subgroup of $\mathrm{Isom}(\R)$, and so $G_v$ cannot be finite or two-ended, so $G_v$ is parabolic.  Since no $P\in \mcP$ can equal $G$, we have $G_v\neq G$, and so $T_v \neq T_\infty$.  

In particular, there is some point $x\in T_v$ where another tree is attached.  Let $e$ be the edge in the skeleton $S$ between the vertex corresponding to $x$ and the one corresponding to $T_v$, and let $E$ be the stabilizer of $e$.  Note that $E$ contains $K$ as a subgroup of index at most two, and that $G$ splits over $E$.  Since $G$ is one-ended (Assumption \ref{assumption:shortest}.\eqref{a:G 1-ended}), $K$ (and hence $E$) must be infinite parabolic.

By contracting edges of $S$ not in the orbit of $e$ to points, we obtain a Bass-Serre tree $\bar{S}$ for the splitting of $G$ over $E$.  By Lemma \ref{lem:cyclic is elliptic}, multi-ended elements of $\mcP$ act elliptically on $T_\infty$, hence on $S$ (see Corollary \ref{cor:tree decomp}), and hence on $\bar{S}$.  By Lemma \ref{lem:make peripheral elliptic}, there is a splitting of $G$ over a parabolic subgroup in which all elements of $\mcP$ are elliptic.  Therefore by Proposition \ref{prop:edge group parabolic} there is a proper peripheral splitting, in contradiction to Assumption \ref{assumption:shortest}.\eqref{a:G 1-ended}.
This implies that $G \acts T_\infty$ has no axial components, as required.
\end{proof}
Thus, every component of the decomposition of $G\acts T_\infty$ is simplicial or of surface type.

Given the finite generating set $\mc{A}$ for $G$ and the basepoint $x_\infty$ of $T_\infty$, we consider the arcs in $T_\infty$ of the form $[x_\infty, a.x_\infty]$ for $a \in \mc{A}$.  Since there is no global fixed point, these arcs are not all trivial.

There are two cases.  The first is when at least one of these arcs intersects a surface type component.  The second is when all of the arcs are contained in (unions of) simplicial components.

We deal with surface-type components first.
\begin{lemma} \label{lem:faithful on surface}
Suppose that $Y_v$ is a surface-type component of the graph of actions decomposition of $T_\infty$, and that $G_v$ is the stabilizer in $G$ of $Y_v$.  The kernel $N_v$ of the action of $G_v$ on $Y_v$ has order at most $C$.
\end{lemma}
\begin{proof}
Any such $Y_v$ contains a nontrivial tripod, so any element of $N_v$ must fix a tripod.  Thus $N_v$ has order at most $C$ by Lemma \ref{lem:tripod stabilizers}.
\end{proof}

The following lemma describes the structure of the splitting of $G$ induced by a surface-type vertex tree in the decomposition of $Y_v$, and is important for proving below that the shortening automorphisms constructed are in $\Aut_{\mc{P}}(G)$.

\begin{lemma}\label{lem:surface vertex}
Suppose that the graph of actions decomposition of $G$ acting on $T_\infty$ contains a surface-type vertex tree $Y_v$ and vertex group $G_v$, with corresponding $2$-orbifold $\Sigma$.  Then
\begin{enumerate}
\item Each attaching point for $Y_v$ (to other sub-trees in the decomposition) corresponds to a boundary component of the universal cover $\tilde\Sigma$, yielding a splitting of $G$ with edge group the subgroup corresponding to this boundary component;
\item Each boundary component of $\Sigma$ arises in this way.
\item\label{hyperbolics in boundary components} Each two-ended subgroup of $G_v$ corresponding to a boundary component contains an RH-hyperbolic element of $G$.
\item Any nontrivial parabolic element contained in $G_v$ has finite order and either corresponds to an orbifold point or is in the kernel of the action of $G_v$ on $T_v$.
\end{enumerate}
\end{lemma}
\begin{proof}
By Lemma \ref{lem:faithful on surface} the map from $G_v$ to $\pi_1\Sigma$ has kernel of order at most $C$.
If there is an attaching point for $Y_v$ that corresponds to a point other than a boundary component, the stabilizer in $\pi_1\Sigma$ is finite, and this would lead to a splitting of $G$ over a finite subgroup, contrary to Assumption \ref{assumption:shortest}.\eqref{a:G 1-ended}.

If there is an `unused' boundary component, an essential arc from this component back to itself would also yield a splitting of $G$ over a finite subgroup.

Suppose that $B$ is a two-ended subgroup of $G_v$ corresponding to a boundary component of $\Sigma$, and suppose that $b \in B$ is infinite order.  We have to prove that $b$ is RH-hyperbolic.  Suppose instead that $b$ is parabolic.  Since parabolic subgroups are almost malnormal, this implies that in fact $B$ is parabolic.  However, the graph of actions of $G \acts T_\infty$ gives a splitting of $G$ over $B$.  Now, every multi-ended $P\in \mcP$ acts elliptically by Lemma \ref{lem:cyclic is elliptic}.  Lemma \ref{lem:make peripheral elliptic} says there is a splitting over a parabolic subgroup where all $P\in \mcP$ are elliptic.  Finally Proposition \ref{prop:edge group parabolic} gives a proper peripheral splitting, which contradicts Assumption \ref{assumption:shortest}.\eqref{a:G 1-ended}.  Thus $b$ must be RH-hyperbolic. 

Finally, suppose that $p$ is a nontrivial parabolic element contained in $G_v$.  Since the only torsion elements of $G_v$ are in the kernel of the $G_v$-action on $T_v$ or else correspond to orbifold points, we need only rule out the case where $p$ has infinite order.  Suppose that there is such an infinite order parabolic element $p$ contained in $G_v$, and let $P$ be the maximal parabolic subgroup containing $p$.  First note that we have already proved that the (two-ended) boundary subgroups of $G_v$ contain RH-hyperbolic elements, which means that $p$ is not contained in such a subgroup, and also the intersection of such a subgroup with $P$ is finite.  

Theorem \ref{thm:Guirardel} gives an action of $G$ on the simplicial skeleton $S$ of the tree $T_\infty$.
If $P$ is not contained in $G_v$, then the action of $P$ on $S$ gives a nontrivial splitting of $P$ over a finite group (the intersection of $P$ with some boundary subgroup of $G_v$), which implies that $P$ is multi-ended.  However, the multi-ended subgroups of $\mc{P}$ act elliptically on $T_\infty$ by Lemma \ref{lem:cyclic is elliptic}, and so on $S$ by Corollary \ref{cor:tree decomp}.  This implies that $P$ is entirely contained in $G_v$.  However, $G_v$ is a virtually free group, so this implies that $P$ is also virtually free, and so again we know that $P$ acts elliptically on $T_\infty$.  But the only subgroups of $G_v$ that act elliptically on $G_v$ correspond to orbifold points or boundary components, and $p$ is not in either kind of subgroup.  This contradiction implies that there are no infinite order parabolic elements in $G_v$, as required.
\end{proof}

Suppose that an arc of the form $[x_\infty,a.x_\infty]$ intersects a surface vertex tree (for some $a \in \mc{A}$).  In the case that $G$ is torsion-free, Rips--Sela \cite{RipsSela94} explain how to obtain an automorphism of $G$ which shortens the action on $T_\infty$, and therefore shortens all but finitely many of the approximating actions.  In \cite[Theorem 4.15]{ReinfeldtWeidmann}, Reinfeldt--Weidmann adapt this argument in the presence of torsion.   One of the subtleties they must deal with is that not all automorphisms of $\pi_1(\Sigma)$ need induce automorphisms of $G_v$ because the map from $G_v$ to $\pi_1(\Sigma)$ may have finite kernel.  We have the additional requirement that our shortening automorphisms must lie in $\Aut_{\mcP}(G)$.  Both these issues are dealt with in the following slight strengthening of \cite[Lemma 4.17]{ReinfeldtWeidmann}.

\begin{lemma} \label{lem:RW}
 Let $\Gamma$ be a finitely presented group and $\mc{H} = \{ H_1, \ldots , H_k \}$ a finite collection of cyclic, malnormal subgroups of $\Gamma$.  Suppose that there is a short exact sequence
 \[	1 \to E \to \tilde{\Gamma} \stackrel{\pi}{\to} \Gamma \to 1	\]
 where $E$ is finite.  Let $\tilde{\mc{H}} = \{ \pi^{-1}(H_i) \}$.  Let $A'$ be the group of automorphisms $\sigma \in \Aut_{\mc{H}}(\Gamma)$ that (i) lift to $\Aut_{\tilde{\mc{H}}}(\tilde{\Gamma})$ and (ii) such a lift acts as the identity on $E$.  Then $A'$ has finite index in $\Aut_{\mc{H}}(\Gamma)$.
\end{lemma}
\begin{proof}
 Let $A$ be the subgroup of $\Aut_{\mc{H}}(\Gamma)$ that lift to $\Aut_{\tilde{\mc{H}}}(\tilde{\Gamma})$.  Then \cite[Lemma 4.17]{ReinfeldtWeidmann} states that $A$ has finite index in $\Aut_{\mc{H}}(\Gamma)$.  However, each element of $A$ lifts to an element of $\Aut_{\tilde{\mc{H}}}(\tilde{\Gamma})$ which preserves $E$, so we get a homomorphism from $A$ to $\Aut(E)$, which is finite.  Then $A'$ is the kernel of this homomorphism.
\end{proof}
(We apply the lemma with $\tilde{\Gamma} = G_v$, $\Gamma = \pi_1(\Sigma)$, and $\mc{H}$ the boundary subgroups of $\pi_1(\Sigma)$ and finite cyclic subgroups corresponding to cone points of $\pi_1\Sigma$.  Since $E$ may contain parabolic elements, we need our automorphisms to act as the identity on $E$.)

We can now follow the proof of \cite[Proposition 4.16]{ReinfeldtWeidmann} exactly as written, except using $A'$ in place of their subgroup $S$ (which is $A$ in Lemma \ref{lem:RW} above), to find the shortening automorphism of $G_v$, and then extend it to an automorphism of $G$ as in \cite[$\S4$]{ReinfeldtWeidmann} (cf. \cite[Proposition 5.4]{RipsSela94}).  

The extension is most naturally described in terms of a coarsening of the graph of groups decomposition of $G$ coming from $G\acts S$.  Let $\Lambda$ be the quotient of $S$ by the action of $G$, let $\pi\co S\to \Lambda$ be the natural quotient map, and let $\mathrm{st}(v)$ be the open star of $v$ in $\Lambda$.  Let $\bar{S}$ be the tree obtained from $S$ by smashing connected components of $\pi^{-1}(\Lambda\smallsetminus\mathrm{st}(v))$ to points.  This gives a new graph of groups decomposition for $G$, with underlying graph $\bar{\Lambda}$, which still contains the vertex $v$, and has one additional vertex for each component of $\Lambda\smallsetminus\{v\}$.  The edge groups incident to $v$ are the same as those in the original graph of groups.  The following is a consequence of the description of the extension given in \cite{RipsSela94}.
\begin{lemma}\label{lem:partialconjugation}
The extension $\bar{\alpha}\co G_v\to G_v$ of a shortening automorphism $\alpha$ to $G$ satisfies:
\begin{enumerate}
\item $\bar{\alpha}|G_v = \alpha$, and
\item For each $w\neq v$ in $\bar{\Lambda}$, the restriction $\bar{\alpha}|G_w = \ad_g$ for some $g\in G_v$.
\end{enumerate}
\end{lemma}

Though we don't know that parabolics are elliptic on $S$, the next lemma shows they are elliptic on $\bar{S}$.
\begin{lemma}
  Let $P\in \mcP$.  Then $P$ fixes a point of $\bar{S}$.  
\end{lemma}
\begin{proof}
  If not, $\bar{S}$ would give a splitting of $P$ over some $P_0$ contained in  boundary subgroup of $G_v$.  By Lemma \ref{lem:surface vertex}.\eqref{hyperbolics in boundary components}, $P_0$ must be finite.  But then $P$ is multi-ended, so it fixes a point of $T_\infty$, hence (Corollary \ref{cor:tree decomp}) fixes a point of $S$, hence fixes a point of $\bar{S}$.
\end{proof}

Now let $P\in \mcP$, and consider a shortening automorphism $\alpha$ chosen to lie in the subgroup $A'$ of $\Aut_{\tilde{H}}(G_v)$ as in the conclusion of Lemma \ref{lem:RW}.  Let $\bar{\alpha} \in \Aut(G)$ be the extension satisfying the conclusions of Lemma \ref{lem:partialconjugation}.  Since $P$ fixes a point of $\bar{S}$, it is conjugate into a vertex group of the graph of groups $\bar{\Lambda}$.  If the corresponding vertex is not $v$, then Lemma \ref{lem:partialconjugation} implies that $\bar{\alpha}|P$ is $\ad_g$ for some $g$.  If the corresponding vertex is $v$, then $P$ must actually be finite and project to a finite cyclic group in $\pi_1\Sigma$.  Lemma \ref{lem:RW} implies that $\bar\alpha$ restricts to the identity on $P$.  Thus $\bar{\alpha}\in \Aut_{\mcP}(G)$, as required, providing the required contradiction in this case.
To summarize, we have the following:

\begin{lemma}
 The arcs $[x_\infty,a.x_\infty]$ never intersect a surface component in a non-degenerate arc.
\end{lemma}

Therefore, we are left with the case that all of the arcs $[x_\infty,a.x_\infty]$ are covered by arcs from simplicial subtrees.  Since $T_\infty$ is a minimal $G$-tree, this implies that $T_\infty$ is in fact simplicial.  In this case, the proof that for sufficiently large $i$ the action $\eta_i$ can be shortened is very similar to the argument in \cite[Theorem 2.5]{sela:acyl}.  Once again, see \cite[$\S4$]{ReinfeldtWeidmann} for details in the presence of torsion.  As in the case the arcs intersect surface-type pieces, the key is to check that the automorphisms used to shorten are in $\Aut_{\mc{P}}(G)$.  It is also important that the two-ended subgroups found as edge groups in the splitting induced by $T_\infty$ have infinite center, since then there are many Dehn twists with which to shorten.

\begin{lemma} \label{lem:not parabolic}
In case $T_\infty$ is simplicial, edge stabilizers are two-ended subgroups which contain (infinite order) RH-hyperbolic elements.  These two-ended subgroups have infinite center.
\end{lemma}
\begin{proof}
We know that edge stabilizers are small by Corollary \ref{cor:stab small}, and hence elementary.  Since the action is simplicial, $G$ splits over its edge stabilizers.  Therefore, the edge stabilizers are not finite, since $G$ is one-ended.  Also, since multi-ended subgroups of $\mcP$ are elliptic in $T_\infty$ by Lemma \ref{lem:cyclic is elliptic}, and since $G$ admits no proper peripheral splittings, by Assumption \ref{assumption:shortest}.\eqref{a:G 1-ended}, we know that the edge stabilizers are not parabolic.  Hence, they are two-ended and contain infinite order RH-hyperbolic elements.

It remains to see that these two-ended subgroups must have infinite center.  However, this is immediate from the structure of the arc stabilizers in Lemma \ref{lem:arc stab structure}.
\end{proof}

In the simplicial case, the automorphisms are Dehn twists in two-ended subgroups which arise as edge groups in the (simplicial) tree $T_\infty$.  By Lemma \ref{lem:not parabolic} such an edge group must contain an RH-hyperbolic element and have infinite center.  That such automorphisms are in $\Aut_{\mc{P}}(G)$ is the content of Example \ref{Ex:AutPG}, in the case of an amalgam, and is similar in the case of an HNN extension.  

This implies that in fact the action $\eta_i$ arising from the fillings $G \to \bar{G}_i$ are not shortest for large $i$, in contradiction to Assumption \ref{assumption:shortest}.\eqref{a:short}.  Thus, we arrive at the required contradiction, and Theorem \ref{t:maintheorem} is proved.

\section{Fuchsian fillings}\label{s:fuchsian}

This short section deals with the possibility of Dehn fillings which are Fuchsian.  Theorem \ref{t:Fuchsian fillings} below is needed to make certain of the statements in the next section cleaner, and may be interesting in its own right.  We say a group is \emph{Fuchsian} if it is equal to the fundamental group of a hyperbolic orbifold, which is to say it is isomorphic to a discrete subgroup of $\mathrm{Isom}(\bH^2)<O(2,1)$.  It is a consequence of the Convergence Group Theorem \cite{cassonjungreis, gabai:convergence} that every virtually Fuchsian group $G$ fits into a short exact sequence
\[ 1 \to F \to G \to \Gamma \to 1\]
so that $\Gamma$ is Fuchsian and $F$ is finite.

\begin{theorem}\label{t:Fuchsian fillings}
  Suppose that $(G,\mcP)$ is relatively hyperbolic, and every small subgroup of $G$ is finitely generated.  Further suppose that $G$ is not virtually Fuchsian, and admits no small splittings.  Then for all sufficiently long fillings $(\bar{G},\bar{\mc{P}})$ of $(G,\mcP)$, the quotient $\bar{G}$ is not virtually Fuchsian.
\end{theorem}
\begin{proof}
  Consider a longer and longer sequence of fillings $(\bar{G}_i,\bar\mcP_i)$ which are virtually Fuchsian, so that each $\bar{G}_i$ fits into a short exact sequence:
\[ 1 \to F_i\to \bar{G}_i\to \Gamma_i \to 1\]
where $F_i$ is finite and $\Gamma_i$ is Fuchsian.  Note that $F_i$ must be contained in every parabolic of $(\bar{G}_i,\bar\mcP_i)$, so by Corollary \ref{cor:uniformC} the size of $F_i$ is uniformly bounded.  By passing to a subsequence, we can suppose that $F_i$ is the isomorphic image of a fixed finite $F< G$.  
This $F$ is the stable kernel of the sequence of fillings $G\to \bar{G}_i$.  Let $\Gamma = G/F$. 

  Each $\Gamma_i$ has a faithful representation as a discrete subgroup of $O(2,1)$.  Each such representation gives a point in the $O(2,1)$--character variety of $G$.  If some subsequence converges, then $\Gamma$ is Fuchsian, contradicting the assumption that $G$ is not virtually Fuchsian.

  The characters therefore diverge.  
  Picking particular representations $\rho_i\co \Gamma\to \mathrm{Isom}(\bH^2)<O(2,1)$, and rescaling appropriately (by some constants $\lambda_i\to 0$), these representations limit to an action on an $\bR$--tree $T$.  (We make sure to conjugate these representations to always have the same, centrally located basepoint.)

  \begin{claim}\label{claim:elementary}
    Arc stabilizers in $T$ are small.
  \end{claim}
  \begin{claim}\label{claim:stable}
    $T$ is a stable $G$--tree.
  \end{claim}
  From these two claims, it follows from \cite[Theorem 9.5]{BF:stable} that $G$ splits over a small-by-abelian (hence small) subgroup, contradicting the hypothesis of no small splitting.  We finish by proving the claims.
  \begin{proof}[Proof of Claim \ref{claim:elementary}]
    Let $I = [p,q]$ be a nondegenerate arc of $T$.  Let $p_i\in \bH^2$ approximate $p$, and let $q_i\in\bH^2$ approximate $q$.  We have $\lambda_id(p_i,q_i)\xrightarrow[i\to\infty]{}d(p,q)$, but for each $g\in \mathrm{Stab}(I)$, 
\[ \max\{\lambda_id(p_i,gp_i),\lambda_id(q_i,gq_i)\}\xrightarrow[i\to\infty]{}0.\]
    One can then show using an elementary hyperbolic geometry argument that for any $a,b,c,d$ in $\Gamma$, the elements $c_1=[a,b]$ and $c_2=[c,d]$ satisfy $\lim\limits_{i\to\infty}\rho_i(c_i)=\lim\limits_{i\to\infty}\rho_i(c_2) = I\in O(2,1)$.  To see this, note that, for large $i$, the discrete group $\rho_i(\langle c_1,c_2\rangle)$ is generated by elements of a Zassenhaus neighborhood of the identity in $O(2,1)$, and must therefore be abelian (cf. \cite[11.6.14]{Beardon}).  Since the $\rho_i$ are stably faithful, we deduce that the stabilizer of $I$ in $\Gamma$ is metabelian.  Thus the stabilizer of $I$ in $G$ is finite-by-metabelian.  In particular it is small.
  \end{proof}
  \begin{proof}[Proof of Claim \ref{claim:stable}]
    By Claim \ref{claim:elementary}, arc stabilizers are small.  Since small subgroups of $G$ are finitely generated, arc stabilizers satisfy the ascending chain condition.  In particular, the action is stable (see \cite[Proposition 3.2.(2)]{BF:stable}).
  \end{proof}
\end{proof}

\section{Behavior of the Bowditch boundary under filling}\label{s:boundaries}

In this section we apply the main results about splittings and fillings to deduce certain consequences about connectivity properties of the boundary under fillings, in particular Theorems \ref{t:no local cutpoints} and \ref{cor:nocutpoints}.  The statements of results are cleanest if we restrict to virtually polycyclic  parabolics, though the alert reader will see that the hypotheses can be weakened in various ways.

\subsection{Literature review of boundaries and splittings}

In this subsection we recall different results about topological properties of boundaries of relatively hyperbolic groups and splittings.

Connectedness was understood first.  Note that \emph{finite} maximal par\-a\-bol\-ics give rise to isolated points in the Bowditch boundary.
\begin{theorem}\cite[Proposition 10.1]{bowditch:relhyp}\label{bowditchthm:connected}
  Let $(G,\mc{P})$ be relatively hyperbolic, and suppose every element of $\mc{P}$ is infinite.  Then $\partial(G,\mc{P})$ is disconnected if and only if $G$ splits non-trivially over a finite subgroup relative to $\mc{P}$.  
\end{theorem}

If $\partial(G,\mc{P})$ is connected, then Bowditch showed that (global) cut points correspond to peripheral splittings, given some mild conditions on the parabolic subgroups.  
\begin{definition}
  Say a group is \emph{tame} if it is finitely presented and contains no infinite torsion subgroup.
\end{definition}
\begin{theorem} \label{bowditchthm:cutpoint}
  Let $(G,\mc{P})$ be relatively hyperbolic, where the elements of $\mc{P}$ are tame and one or two-ended.  Suppose that $\partial(G,\mc{P})$ is connected.

  Then $\partial(G,\mc{P})$ has a cut point if and only if $(G,\mc{P})$ has a nontrivial peripheral splitting.
\end{theorem}
\begin{proof}
  Suppose first that $(G,\mc{P})$ has a nontrivial peripheral splitting.  Then \cite[Theorem 1.2]{bowditch:peripheral} implies that $\partial(G,\mc{P})$ has a cut point.  (Note that this direction only requires connectedness of the boundary, and not the extra hypotheses on the parabolics.)

  Conversely, suppose that $\partial(G,\mc{P})$ has a cut point.  Theorem 0.2 of \cite{bowditch:connectedness} implies that, under the hypotheses, every cut point of $\partial(G,\mc{P})$ is a parabolic fixed point.  Theorem 1.2 of \cite{bowditch:limit} then implies that $(G,\mc{P})$ admits a proper peripheral splitting.  
\end{proof}

Finally local cut points are connected to splittings over $2$--ended groups.  Recall that a \emph{continuum} is a connected compact Hausdorff space.  To have a reasonable notion of local cut point, we must restrict attention to locally connected spaces.
\begin{theorem}\cite[Theorem 1.5]{bowditch:peripheral}\label{bowditchthm:locallyconnected}
    Let $(G,\mc{P})$ be relatively hyperbolic, where the elements of $\mc{P}$ are tame and one or two-ended, and suppose that $\partial(G,\mc{P})$ is connected.  Then $\partial(G,\mc{P})$ is locally connected.
\end{theorem}
\begin{definition}
  Let $p\in M$ where $M$ is a locally connected continuum.  The \emph{valence}  $\mathrm{Val}(p)$ is the cardinality of $\mathrm{Ends}(M\smallsetminus\{p\})$.  The point $p$ is called a \emph{local cut point} if $\mathrm{Val}(p)>1$.
\end{definition}

\begin{lemma}\label{lem:bppvalence}
  Suppose $(G,\mc{P})$ is relatively hyperbolic and that $\partial(G,\mc{P})$ is a locally connected continuum.
    If $p$ is a parabolic fixed point which is not a cut point, then $\mathrm{Val}(p)$ is equal to the number of ends of $\mathrm{Fix}(p)$.
\end{lemma}
\begin{proof}
Recall that the action of a relatively hyperbolic group on its Bowditch boundary is \emph{geometrically finite} \cite[Proposition 6.15]{bowditch:relhyp}, meaning that every point is either a \emph{conical limit point} or a \emph{bounded parabolic point}.  
Since $p$ is a parabolic fixed point, it is a bounded parabolic point, which means that $\mathrm{Fix}(p)$ acts properly discontinuously and cocompactly on $\partial(G,\mcP)\smallsetminus \{p\}$.  Since $p$ is not a cut point, $\partial(G,\mcP)\smallsetminus \{ p \}$ is connected.  It follows that $\mathrm{Ends}(\mathrm{Fix}(p)) = \mathrm{Ends}(\partial(G,\mcP)\smallsetminus \{p\})$.
\end{proof}
\begin{lemma}\label{lem:clpvalence}
  Suppose $(G,\mc{P})$ is relatively hyperbolic and that $\partial(G,\mc{P})$ is a locally connected continuum.  There is an $N<\infty$ so that if $p\in \partial(G,\mc{P})$ is not a parabolic fixed point, then $\mathrm{Val}(p)<N$.
\end{lemma}
\begin{proof}
  Bowditch proved this in the absolute case $\mc{P}=\emptyset$ \cite[Proposition 5.5]{Bowditch:cutpoints}.  Guralnik \cite[Proposition 4.2]{GuralnikIJAC} points out that essentially the same proof works in the relatively hyperbolic setting.
\end{proof}

\begin{theorem}\label{groffthm}
  Let $(G,\mc{P})$ be relatively hyperbolic, where the elements of $\mc{P}$ are tame and one-ended, and suppose that $\partial(G,\mc{P})$ is connected and has no cut point.  Suppose further that $(G,\mc{P})$ is not virtually Fuchsian.
  
  Then $(G,\mc{P})$ has a local cut point if and only if $G$ splits relative to $\mc{P}$ over a non-parabolic $2$--ended subgroup.
\end{theorem}
\begin{proof}
  Suppose that $(G,\mc{P})$ splits relative to $\mc{P}$ over a $2$--ended subgroup $H$.  Then $\partial H$ is a two point set which separates $\partial(G,\mc{P})$.  To see this, note that $H$ is quasi-isometrically embedded in the cusped space because it is two-ended and not parabolic.  Considering a cusped graph of spaces realizing the splitting of $G$, we see that $H$ coarsely separates the cusped space for $(G,\mc{P})$ into at least two distinct (deep) components.  From this, it is clear that the pair of fixed points for $H$ separates the boundary $\partial(G, \mc{P})$.  

  Conversely, suppose that $(G,\mc{P})$ satisfies the hypotheses of the theorem, and $\partial(G,\mc{P})$ has a local cut point.  Theorem \ref{bowditchthm:locallyconnected} implies that $\partial(G,\mc{P})$ is locally connected, so we may apply Lemmas \ref{lem:bppvalence} and \ref{lem:clpvalence} to conclude there is no infinite valence point in $\partial(G,\mc{P})$.  It follows that $\partial(G,\mc{P})$ is \emph{cut-rigid} in the sense of Guralnik's paper \cite{GuralnikIJAC}.
Propositions 4.7 and 4.8 of \cite{GuralnikIJAC} can then be used to conclude that $\partial(G,\mc{P})$ has a cut pair.  

  The work of Papasoglu--Swenson in \cite{PapSwe06,PapSwe11} encodes the cut-pair structure of the boundary in a pre-tree which canonically (and therefore $G$--equivariantly) embeds in a simplicial tree $T$.  
Groff \cite[Theorem 5.1]{Groff13} shows this tree is a JSJ tree for elementary splittings of $(G,\mc{P})$, relative to $\mc{P}$.\footnote{Groff requires that the elements of $\mc{P}$ are not themselves properly relatively hyperbolic, but does not use this hypothesis for this result.  This assumption is needed at other points in \cite{Groff13}.}  Since $\partial(G,\mc{P})$ is connected, there is no splitting over a finite group relative to $\mc{P}$ (Theorem \ref{bowditchthm:connected}); since $\partial(G,\mc{P})$ has no cut point, and there are no multi-ended elements of $\mc{P}$, there is no splitting over a parabolic subgroups (Lemma \ref{lem:make peripheral elliptic} and Theorem \ref{bowditchthm:cutpoint}).  Thus every edge stabilizer comes from a splitting over a non-parabolic $2$--ended subgroup relative to $\mc{P}$.  

It remains to establish that the action of $G$ on $T$ has no global fixed point.  Here we use Groff's explicit description of the tree \cite[Section 4]{Groff13}: vertices correspond to
\begin{enumerate}
\item cut points (we do not have any of these),
\item inseparable cut pairs,
\item necklaces, or
\item equivalence classes of points not separated by any cut point or cut pair.
\end{enumerate}
See \cite{Groff13} for the definition of these terms.
Edges correspond to intersection of closures.  If there is a global fixed point in $T$, one quickly sees it must correspond to a \emph{necklace}, which equals $\partial(G,\mc{P})$.  But in this case $\partial(G,\mc{P}) = S^1$ and the pair $(G,\mc{P})$ is virtually Fuchsian, by the Convergence Group Theorem \cite{cassonjungreis, gabai:convergence}.
\end{proof}

\begin{corollary}\label{cor:elementarysplitting}
  Let $(G,\mc{P})$ be relatively hyperbolic, so elements of $\mcP$ are tame and one-ended.  Suppose either $\mcP$ is nonempty or $G$ is hyperbolic but not virtually Fuchsian.
  The following are equivalent:
  \begin{enumerate}
  \item $\partial(G,\mc{P})$ is connected and has no (local or global) cut point.
  \item $(G,\mc{P})$ has no elementary splitting.
  \end{enumerate}
\end{corollary}
\begin{proof}
  This is immediate from Theorems \ref{bowditchthm:connected}, \ref{bowditchthm:cutpoint}, and \ref{groffthm}.
\end{proof}

\subsection{Connectedness properties of the boundary after filling}
Now we prove the statements about fillings and boundaries from the introduction, first recalling some notation.
For $(G,\mc{P})$ a relatively hyperbolic group, let $\mc{P}^{\infty}\subseteq \mc{P}$ be the collection of infinite peripheral subgroups, and let $\mc{P}^{\mathrm{red}}$ be the collection of non-hyperbolic subgroups.  Then  $(G,\mc{P}^{\infty})$ and $(G,\mc{P}^{\mathrm{red}})$ are still relatively hyperbolic, but the Bowditch boundaries may be different.  For example, if $(G,\mc{P}^{\mathrm{red}})$ is connected, but $\mc{P}\smallsetminus \mc{P}^{\mathrm{red}}$ contains an infinite cyclic subgroup, then $\partial(G,\mc{P})$ has a local cut point.  The Bowditch boundary of $(G,\mc{P}^{\infty})$ is equal to $\partial (G,\mc{P})$ with its isolated points removed.

Also recall the inductive definition of \emph{polycyclic} groups:  Cyclic groups are polycyclic.  Moreover a group $E$ is polycyclic whenever it fits into a short exact sequence
\[ 1\to P\to E\to C \to 1\]
with $C$ cyclic and $P$ polycyclic.  In particular, finitely generated nilpotent groups are polycyclic.
To apply our results in this setting we need a couple of observations about this class of groups.
\begin{lemma}\label{lem:polycyclic facts}
  Let $P$ be virtually polycyclic.  Then $P$ is tame, slender, and one or two--ended.  Moreover every quotient of $P$ is virtually polycyclic.
\end{lemma}

Recall the statement of Theorem \ref{cor:nocutpoints}.
\nocutpoints*
\begin{proof}
  By Lemma \ref{lem:polycyclic facts},  the peripheral subgroups $\mc{P}$ satisfy the hypotheses of Bowditch's Theorem \ref{bowditchthm:cutpoint}.  There is no cut point in $\partial(G,\mcP)$, so $(G,\mc{P})$ has no proper peripheral splitting.

  Let $(\bar{G},\bar{\mc{P}})$ be sufficiently long that we can apply Theorem \ref{t:maintheorem}, so that $\bar{G}$ is one-ended, and 
  $(\bar{G},\bar{\mc{P}})$ has no proper peripheral splittings. 

  It may be the case that $\partial(\bar{G},\bar{\mc{P}})$ is disconnected, but this can only be because some elements of $\bar{\mc{P}}$ are finite.  All the elements of $\bar{\mc{P}}^\infty$ are infinite, so Theorem \ref{bowditchthm:connected} implies that $\partial (\bar{G},\bar{\mc{P}}^\infty)$ is connected.
Since there is no nontrivial peripheral splitting of $(\bar{G},\bar{\mc{P}})$, it is easy to see there is also no nontrivial peripheral splitting of $(\bar{G},\bar{\mc{P}}^{\infty})$.  Theorem \ref{bowditchthm:cutpoint} then implies that $\partial (\bar{G},\bar{\mc{P}}^{\infty})$ has no cut point.
\end{proof}

We now prove Theorem \ref{t:no local cutpoints}.
\nolocalcutpoints*
\begin{proof}
The hypothesis implies $G$ is not Fuchsian, since then $\partial(G,\mcP^{\mathrm{red}})$ would either be a Cantor set or $S^1$.  Note that since $\mc{P}$ consists of virtually polycyclic groups, $\mc{P}^{\mathrm{red}}$ consists of one-ended groups.  Corollary \ref{cor:elementarysplitting} implies that $(G,\Pred)$ has no elementary splittings.  It follows that $(G,\mcP)$ has no elementary splittings, since the elementary subgroups of $(G,\mcP)$ and $(G,\Pred)$ coincide.  Now let $(\barG,\barP)$ be a filling which is sufficiently long that Theorems \ref{t:easy maintheorem} and \ref{t:Fuchsian fillings} apply.  Since $(\barG,\barP)$ has no elementary splittings, neither does $(\barG,\barP^{\mathrm{red}})$.  Moreover Theorem \ref{t:Fuchsian fillings} implies $\barG$ is not Fuchsian.  Corollary \ref{cor:elementarysplitting} then implies that $\partial(\barG,\barP^{\mathrm{red}})$ is connected with no local cut points.
\end{proof}

\end{document}